\documentclass[12pt]{amsart}

\usepackage{amsmath,bm}
\usepackage[color=yellow,textwidth=2cm,colorinlistoftodos]{todonotes}
\usepackage{hyperref}
\usepackage[nameinlink,capitalize]{cleveref}
\usepackage{dsfont}
\usepackage[english]{babel}

\newcommand{\summe}[2]{\displaystyle\sum_{#1}^{#2}}
\newcommand{\summezwei}[2]{\sum_{#1}^{#2}}

\newcommand{\integral}[4]{\displaystyle\int\limits_{#1}^{#2}{#3}\,{#4}}

\newcommand{\indikator}[2]{\mathds{1}_{{#1}}({#2})}
\newcommand{\indikatorzwei}[1]{\mathds{1}_{{#1}}}

\newcommand{\bigerwartung}[1]{\mathbb{E}\left({#1}\right)}

\newcommand{\konv}[2]{\xrightarrow[({#1}\rightarrow{#2})]{}}

\newcommand{\konviv}{\Longrightarrow}

\newcommand{\limes}[2]{\lim \limits_{#1 \rightarrow #2}}

\newcommand{\ft}[1]{\widehat{{#1}}}
\newcommand{\B}{\mathcal{B}}
\newcommand{\im}{\mathrm{i}}
\newcommand{\C}{\mathbb{C}}

\newcommand{\R}{\mathbb{R}}
\newcommand{\Q}{\mathbb{Q}}

\newcommand{\N}{\mathbb{N}}

\newcommand{\Pro}{\mathbb{P} } 
\newcommand{\eps}{\varepsilon} 
\newcommand{\secret}[1]{}

\newcommand{\norm}[1]{\| {#1}\|}
\newcommand{\Li}[1]{\text{L}(\R^{#1})}

\newcommand{\bignorm}[1]{ \left \Vert {#1} \right \Vert}
\newcommand{\skp}[2]{ \langle {#1},{#2} \rangle}

\newcommand{\skpzwei}[2]{ \left \langle {#1},{#2} \right \rangle}

\newtheorem{theorem}{Theorem}[section]
\newtheorem{lemma}[theorem]{Lemma}
\newtheorem{prop}[theorem]{Proposition}
\newtheorem{cor}[theorem]{Corollary}

\theoremstyle{definition}
\newtheorem{defi}[theorem]{Definition}
\newtheorem{example}[theorem]{Example}

\theoremstyle{remark}
\newtheorem{remark}[theorem]{Remark}

\numberwithin{equation}{section}
\begin{document}
\sloppy
\title[]{Operator-stable and operator-self-similar random fields}

\author{D. Kremer}
\address{Dustin Kremer, Department Mathematik, Universit\"at Siegen, 57068 Siegen, Germany}
\email{kremer\@@{}mathematik.uni-siegen.de}

\author{H.-P. Scheffler}
\address{Hans-Peter Scheffler, Department Mathematik, Universit\"at Siegen, 57068 Siegen, Germany}
\email{scheffler\@@{}mathematik.uni-siegen.de}

\date{\today}

\begin{abstract}
Two classes of multivariate random fields with operator-stable marginals are constructed. The random fields $\mathbb{X}=\{X(t):t \in \R^d\}$ with values in $\R^m$ are invariant in law under operator-scaling in both the time-domain and the state-space. The construction is based on operator-stable random measures utilising certain homogeneous functions.
\end{abstract}
\maketitle
\section{Introduction}
The notion of self-similarity of stochastic processes and random fields yields to a rich class of stochastic models with application in various fields, such as physics, ground water hydrology and mathematical finance, just to mention a few. See for instance \cite{Abry}, \cite{Embrechts} and \cite{Vehel}. \\
Recall from \cite{bms} that a scalar-valued random field $\mathbb{X}=\{X(t):t \in \R^d\}$ is called \textit{operator-scaling}, if for some $d \times d$ matrix $E$ with positive real parts of the the eigenvalues we have
\begin{equation} \label{eq:22111701}
\{X(r^E t): t \in \R^d \} \overset{fdd}{=} \{r X( t): t \in \R^d \}
\end{equation}
for all $r>0$. Here $\overset{fdd}{=}$ denotes equality of all finite-dimensional marginal distributions. As usual $r^E=\exp ((\ln r)E)$ is the matrix exponential, see \eqref{eq:22111710} below. Observe that if $E= \frac{1}{H} I_d $, where $I_d$ is the identity matrix, then \eqref{eq:22111701} is just the well-known self-similarity property with \textit{Hurst-index} $H>0$ as first studied in \cite{lamperti}. \\
In \cite{bms} two  usually different representations of operator-scaling random fields with symmetric $\alpha$-stable ($S \alpha S$) marginals are presented. Those representations are based on $S \alpha S$  random measures and so called $E$-homogeneous functions. More precisely, for $0 < \alpha \le 2$ let $M_\alpha$(ds) be an independently scattered $S \alpha S$ random measure on $\R^d$ with Lebesgue control measure and $\phi$ be an $E$-homogeneous  function (see chapter 4 below for a definition). Then with $q=trace(E)$ the random field defined by the \textit{moving-average} representation  
\begin{equation} \label{eq:16051701}
X(t)=\integral{\R^d}{}{ \left( \phi(t-s)^{1-\frac{q}{\alpha} } - \phi(-s)^{1-\frac{q}{\alpha} } \right) }{M_{\alpha}(ds)}
\end{equation} 
satisfies \eqref{eq:22111701}. For a complex-valued isotropic $S \alpha S$ random measure $\widetilde{M}_{\alpha}(ds)$ with Lebesgue control measure the \textit{harmonizable} representation is given by 
\begin{equation}  \label{eq:16051702}
\widetilde{X}(t)=\text{Re } \integral{\R^d}{}{ \left( \text{e}^{\im \skp{t}{s}}-1 \right) \phi(s)^{-1-\frac{q}{\alpha} } }{\widetilde{M}_{\alpha}(ds)}
\end{equation}
and again satisfies \eqref{eq:22111701}.  Note that as shown in \cite{bms} both fields are stochastically continuous and have stationary increments. \\
In the multivariate case, that is for $\R^m$-valued random fields $\mathbb{X}=\{X(t):t \in \R^d\}$ much less is known. To our knowledge, a first rigorous investigation of those fields was initiated in \cite{LiXiao}. Since $X(t)$ is $\R^m$-valued it is usual to scale both time $t \in \R^d$ and the space $\R^m$ by matrices. It is shown in Theorem 2.2. of \cite{LiXiao} that under rather general conditions there exists a $m \times m$ matrix $D$, called the space exponent, such that 
\begin{equation} \label{eq:22111705}
\{X(r^E t): t \in \R^d \} \overset{fdd}{=} \{r^D X( t): t \in \R^d \}
\end{equation}
for all $r>0$ holds true. $\mathbb{X}$ is then called an $(E,D)$-operator-self-similar random field. As in the scalar valued case, in \cite{LiXiao} a moving-average and a harmonizable representation  based on multivariate $S \alpha S$ random measures and certain $E$-homogeneous functions are presented. \\
Multivariate stable laws are a special case of operator-stable laws which allow much more modelling flexibility. In fact for multivariate stable laws the tail behavior is equal in every direction, whereas operator-stable laws usually have different tail behavior in different directions. See \cite{thebook} for a comprehensive introduction to operator-stable laws and their limit theorems. \\
The purpose of this paper is to give both a moving-average and a harmonizable representation of an $(E,D)$-operator-self-similar random field with operator-stable marginals. Using the recently developed theory of multivariate infinitely-divisible random measures in \cite{integral} we first construct operator-stable random measures in chapter 2 of this paper. It is well-known that (operator) self-similarity is strongly connected to limit theorems. In chapter 3 we investigate so called domains of attraction of random fields and show that this naturally leads to $(E,D)$-operator-self-similar random fields. Then in chapter 4 and 5 the moving-average and harmonizable representation of such fields with operator-stable marginals is given and some basic properties are analyzed.
\section{Operator-stable distributions and random measures} \label{chapter2}
Let (G)L$(\mathbb{\R}^m)$ be the set of all (invertible) linear operators on $\R^m$, represented as $m \times m$ matrices.  Then, with a little abuse of notation, $\norm{\cdot}$ denotes the Euclidian norm on $\R^m$ (with inner product $\skp{\cdot}{\cdot}$) as well as the operator norm on L$(\mathbb{\R}^m)$ which is induced by the previous vector space norm. As usual the matrix exponential of $A \in \Li{m}$ is defined via
\begin{equation} \label{eq:22111710}
s^{A}:=\exp ( (\ln s)A):=\summe{j=0}{\infty} \frac{(\ln s)^j}{j!} A^j  \quad \in \text{GL}(\R^m)
\end{equation} 
for every $s>0$ with $A^{0}:=I_m$ being the identity operator on $\R^m$. See chapter 2.2 in \cite{thebook} for a comprehensive summary of the most important properties. Let
\begin{equation*}
\lambda_A:=\min \{\text{Re } \lambda : \lambda \in \text{spec}(A)\}, \quad \Lambda_A:=\max \{\text{Re } \lambda : \lambda \in \text{spec}(A)\},
\end{equation*}
where $\text{spec}(A)$ denotes the spectrum of $A \in \Li{m}$, i.e. the set of all complex roots of the characteristic polynomial. Then it follows from Theorem 2.2.4 in \cite{thebook} (for sharper estimates see Lemma 3.2 in \cite{BiLa09}) that
\begin{equation} \label{eq:17051701}
\forall s_0,\delta>0 \, \, \exists C_1,C_2>0: \quad \norm{s^A} \le \begin{cases} C_1 s^{\lambda_A-\delta}, & 0<s \le s_0, \\ C_2 s^{\Lambda_A+\delta}, & s \ge s_0.  \end{cases}
\end{equation}
Finally let $Q(\R^m):=\{A \in \text{L}(\R^m):\lambda_A>0\}$. It then follows from Lemma 6.1.5 in \cite{thebook} that for $A \in Q(\R^m)$ there exists a norm $\norm{\cdot}_A$ on $\R^m$ such that $ (s,\theta) \mapsto s^{A} \theta$ defines a homeomorphism $\Psi:(0,\infty) \times S_A \rightarrow \Gamma_m:=\R^m \setminus \{0\}$, where $S_A=\{x \in \R^m: \norm{x}_A=1\}$. Then we call $(\tau_{(A)}(x),l_{(A)}(x)):=\Psi^{-1}(x)$ the \textit{(generalized) polar coordinates} for $x \in \Gamma_m$ w.r.t $A$. Note that $S_A=\{x: \tau(x)=1\}$ as well as $\tau(x)=\tau(-x)$ and $\tau(s^{A}x)= s \tau(x)$ by uniqueness. The following relation between $\tau(\cdot)$ and $\norm{\cdot}$ is given by Lemma 2.1 of \cite{bms}. Hence for every $\delta>0$ and $s_0>0$ there exist $C_1,C_2,C_3,C_4>0$ such that
\begin{align}
C_1 \norm{x}^{1/ \lambda_A+\delta} \le & \tau(x) \le C_2 \norm{x}^{1/ \Lambda_A-\delta}, \quad \text{if $\tau(x) \le s_0$,}    \label{eq:18051701} \\
C_3 \norm{x}^{1/ \Lambda_A-\delta} \le & \tau(x) \le C_4 \norm{x}^{1/ \lambda_A+\delta}, \quad \text{if $\tau(x) \ge s_0$.}   \label{eq:18051702} 
\end{align}
Turning over to probability,  it is well-known that $\varphi=\exp(\psi)$ with $\psi:\R^m \rightarrow \C$ is the Fourier transform (or characteristic function) of an infinitely-divisible distribution on $\R^m$ if and only if $\psi$ can be represented as
\begin{equation*}
\psi(u)=\im  \skp{\gamma}{u} - \frac{1}{2} \skp{Qu}{u} + \integral{\R^m}{}{\left( \text{e}^{\im \skp{u}{x}} - 1 - \frac{\im \skp{u}{x} }{1+\norm{x}^2} \right)}{\nu(dx)} , \quad u \in \R^m
\end{equation*}
for a \textit{shift} $\gamma \in \R^m$, some \textit{normal component} $Q \in \Li{m}$ which is symmetric and positive semi-definite and a L\'{evy} measure $\nu$, i.e. $\nu$ is a measure on $\R^m$ with $\nu(\{0\})=0$ and $\int_{\R^m} \min\{1,\norm{x}^2\} \, \nu(dx)< \infty$. For the distribution $\mu$ with $\ft{\mu}=\varphi$ we write $\mu \sim [\gamma, Q, \nu]$ as $\gamma, Q$ and $\nu$ are uniquely determined by $\mu$. $\psi$ is the only continuous function with $\psi(0)=0$ and $\ft{\mu}=\exp(\psi)$, subsequently referred to as the \textit{log-characteristic function} of $\mu$. In view of Theorem 3.4.1 in \cite{thebook} we can generally define $f^s$ for $s>0$, whenever $f:\R^m \rightarrow \C \setminus \{0\}$ is continuous with $f(0) \in \R_{+}$. In this case $f^s$ describes the Fourier transform of an infinitely-divisible distribution $\mu^{s}$ as long as $f$ itself is the Fourier transform of an infinitely-divisible distribution $\mu$. \\ Next we consider a sequence $X,X_1,X_2,...$ of independent and identically distributed (i.i.d.) $\R^m$-valued random vectors. Then $X$ (or the distribution $\mathcal{L}(X)$ of $X$) is called \textit{operator-stable}, if there exist operators $A_n \in$ GL$(\R^m)$ and \textit{shifts} $a_n \in \R^m$ such that
\begin{equation} \label{eq:18051704}
A_n(X_1+ \cdots + X_n)+a_n \overset{d}{=}X, \quad n \in \N,
\end{equation}
where $\overset{d}{=}$ denotes equality of the corresponding distributions. If additionally $a_n=0$ holds, then $X$ is called \textit{strictly} operator-stable. Often we can show \eqref{eq:18051704} for $A_n=n^{-B}$ and call $X$ operator-stable with \textit{exponent} $B$. This is motivated by the following characterization which is due to Theorem 7.2.1 in \cite{thebook}. Here we say that a distribution on $\R^m$ is \textit{full}, if it is not concentrated on any hyperplane of $\R^m$. 
\begin{theorem} \label{18051710}
A full probability measure on $\R^m$ is operator-stable if and only if it is infinitely-divisible such that there are an operator $B \in Q(\R^m)$ as well as vectors $(a_s)_{s>0} \subset \R^m$ sucht that the following relation holds in terms of the Fourier transform $\ft{\mu}$:
\begin{equation} \label{eq:18051706}
\ft{\mu}(u)^s =\ft{\mu}(s^B u) \exp(\im \skp{a_s}{u}) \quad \text{for all $s>0$ and $u \in \R^m$.}
\end{equation}
Then $\text{spec}(B)$ is uniquely determined by $\mu$, but not the so called exponent $B$ itself. Furthermore, we generally have $\lambda_B \ge \frac{1}{2}$, but $\lambda_B> \frac{1}{2}$ if and only if $\mu$ has no Gaussian component.
\end{theorem}
Observe that our definition of operator-stability coincides with that in \cite{thebook}, if the considered distribution is full. Also note that Definition 3.3.24 in \cite{thebook} accentuates the theoretical importance of operator-stable laws by means of the term \textit{domain of attraction}. 
Overall Lemma 2.2 in \cite{kms} can help to observe that a full operator-stable distribution is strictly operator-stable if and only if we can choose $a_s=0$ in \eqref{eq:18051706}. For our purpose strictly operator-stable distributions will be most important, hence it is mentionable that,  for $\mu$ as in \cref{18051710}, we can always find a $x_0 \in \R^m$ such that $\mu * \eps_{x_0}$ (convolution with the Dirac measure in $x_0$) is strictly operator-stable, at least if $1 \notin \text{spec}(B)$ (see Corollary 4.9.3 in \cite{JuMa93}). The following properties are easy conclusions of each other or can be checked by the previous statements. 
\begin{cor} \label{24051703}
Let $\mu$ be an operator-stable distribution on $\R^m$ which is neither full nor strict a-priori.
\begin{itemize}
\item[(i)] Assume that \eqref{eq:18051706} holds, where $\mu \sim [\gamma,Q,\nu]$ with log-characteristic function $\psi$. Then we have $s \cdot Q=s^{B} Q s^{B^{*}}$ and $s \cdot \nu=(s^B \nu)$ for every $s>0$. Here ${}^{*}$ denotes the adjoint operator and $(s^B \nu)(A)=\nu(s^{-B} A)$ denotes the image measure. Furthermore, the shifts are given by
\begin{equation*}
a_s=(s^{I_m}-s^{B}) \gamma - \integral{\R^m}{}{\left(  \frac{s^B x}{1+\norm{s^Bx}^2} - \frac{s^B x}{1+\norm{x}^2} \right) }{\nu(dx)}.
\end{equation*}
Finally the following relation holds for every $s>0$ and $u \in \R^m$:
\begin{equation} \label{eq:18051711}
s \cdot \psi(u)=\psi(s^{B^{*}}u) + \im \skp{a_s}{u}.
\end{equation}
\item[(ii)] If $\mu$ is symmetric, then it is even strictly operator-stable. The converse is false.
\end{itemize}
\end{cor}
Roughly speaking, the \textit{$\alpha$-stable} distributions are those operator-stable ones which only allow operators of the form $A_n= c_n I_m$ in \eqref{eq:18051704} and therefore $B=c I_m$ in \eqref{eq:18051706}, where $c,c_1,...$ are positive numbers. Writing $c=\frac{1}{\alpha}$, we call $\alpha$ the stability-index and according to \cref{18051710} it follows that $\alpha \in (0,2]$. See \cite{SaTaq94} for a comprehensive overview about stable distributions.  \\
Now we can apply the results in \cite{integral}: Assume $\mu$ subsequently to be a full infinitely-divisible distribution on $\R^m$ with triplet $[\gamma',Q',\nu']$ and log-characteristic function $\psi$, where $(S, \Sigma, \Theta)$ is any non-trivial, $\sigma$-finite measure space. Then we consider the $\R^m$-valued independently scattered random measure $M$ on the $\delta$-ring $\mathcal{S}:=\{A \in \Sigma: \Theta(A)< \infty\}$ (and some suitable probability space $(\Omega,\mathcal{A},\Pro)$) which is generated by $\mu$ and $\Theta$ in the sense of Example 3.7 (a) in \cite{integral}. This means that we have $M(A) \sim [\Theta(A) \gamma' , \Theta(A) Q', \Theta(A) \nu']$ for every $A \in \mathcal{S}$. Concerning the question, whether a given mapping $f:S \rightarrow$ L$(\R^m)$ is integrable with respect to $M$, i.e. belongs to $\mathcal{I}(M)$, we can first focus on the real-valued perspective due to Proposition 5.10 in \cite{integral}.
\begin{theorem} \label{24051701}
For $f:S \rightarrow \Li{m}$ measurable the following statements are equivalent:
\begin{itemize}
\item[(i)] $f \in \mathcal{I}(M)$.
\item[(ii)] The below-mentioned integrals exist:
\begin{align*}
\gamma_f: &= \integral{S}{}{\left(f(s) \gamma' + \integral{\R^m}{}{\left( \frac{f(s)x}{1+\norm{f(s)x}^2} - \frac{f(s)x}{1+\norm{x}^2} \right)}{\nu'(dx)} \right)}{\Theta(ds)}, \\
Q_f:&=\integral{S}{}{f(s)Q' f(s)^{*}}{\Theta(ds)}, \\
L_f:&=\integral{S}{}{\integral{\R^m}{}{\min \{1,\norm{f(s)x}^2\}}{\nu'(dx)} }{\Theta(ds)}.
\end{align*}
\item[(iii)] The integral $\int_S \psi(f(s)^{*}u) \, \Theta(ds)$ exists for every $u \in \R^m$ and the mapping
\begin{equation*}
\R^m \ni u \mapsto \integral{S}{}{\integral{\R^m}{}{(\cos \skp{f(s)^{*}u}{x}-1)}{\nu'(dx)}}{\Theta(ds)}
\end{equation*}
is continuous.
\end{itemize}
\end{theorem}
\begin{proof}
Combine Example 3.7 (a) and Theorem 5.8 in \cite{integral}.
\end{proof}
In order to estimate the integrals in \cref{24051701} for an operator-stable $\mu \sim [\gamma',Q',\nu']$ with exponent $B$ the following desintegration formula for the L\'{e}vy measure $\nu'$ will be useful. It is shown in Theorem 7.2.5 in \cite{thebook} that there exists a finite measure $\sigma$ on $\B(S_B)$, i.e. the collection of all Borel sets $A \subset S_B$, such that for any measurable function $g:\R^m \rightarrow [0,\infty)$ we have
\begin{equation} \label{eq:23051701}
\integral{\R^m}{}{g(x)}{\nu'(dx)}=\integral{S_B}{}{\integral{0}{\infty}{g(r^B \zeta) r^{-2}}{dr} }{\sigma(d \zeta)}.
\end{equation}
\begin{example} \label{30051702}
Let $f:S \rightarrow \Li{m}$ be measurable and assume that $\mu$ is a full symmetric $\alpha$-stable ($S \alpha S$) measure with corresponding spectral measure $\sigma$ on $S_{\alpha}:=S_{\alpha^{-1}I_m}$.
\begin{itemize}
\item[(a)] In case $\alpha=2$ we have $f \in \mathcal{I}(M)$ if and only if $\int_S \norm{f(s)}^2 \, \Theta(ds)< \infty$.
\item[(b)] In case $\alpha<2$ we have $f \in \mathcal{I}(M)$ if and only if 
\begin{equation} \label{eq:23051705}
\integral{S}{}{\integral{S_{\alpha}}{}{\norm{f(s)\zeta}^{\alpha} }{\sigma(d\zeta)} }{\Theta(ds)}< \infty.
\end{equation}
Hence $\int_S \norm{f(s)}^{\alpha} \, \Theta(ds)<\infty$ is s sufficient condition for $f \in \mathcal{I}(M)$ again and even necessary, if $\psi(\cdot)=-c \norm{\cdot}^{\alpha}$ for some $c>0$. 
\end{itemize}
\end{example}
\begin{proof} 
The symmetry of $\mu$ and \cref{18051710} imply that we have either $\mu \sim [0,Q',0]$ or $\mu \sim [0,0,\nu']$ with $\nu'$ being symmetric. In any case $\gamma_f=0$ exists. Concerning part (a) and according to \cref{24051701} (ii) this means that $f \in \mathcal{I}(M)$ if and only if $Q_f$ exists. Hence, since $\norm{f(s)^{*}}=\norm{f(s)}$, we see that $\int_S \norm{f(s)}^2 \, \Theta(ds)< \infty$ is a sufficient condition. Conversely, we assume the existence of $Q_f$ and see by fullness of $\mu$ that $Q'$ even has to be positive definite (and still symmetric). Then we first consider $Q'=\text{diag}(q_1',...,q_m')$ with $q_i'>0$ and suppose that $\int_S \norm{f(s)}^2 \, \Theta(ds)=\infty$. It then follows that there has to be a component $f_{i_0,j_0}$ of $f$ with $\int_S f_{i_0,j_0}(s)^2 \, \Theta(ds)=\infty$. At the same time we know that 
\begin{equation*}
0 \le \skp{Q_f u}{u}=\int_S \skp{Q' f(s)^{*}u}{f(s)^{*}u} \, \Theta(ds)< \infty,
\end{equation*}
which yields the contradiction for $u=e_{i_0}$ (i.e. the $i_0$-th unit vector), since
\begin{equation*}
\skp{Q' f(s)^{*}u}{f(s)^{*}u} = \summe{i=1}{m} q_i' f_{i_0,i}(s)^2 \ge q_{j_0}' f_{i_0,j_0}(s)^2
\end{equation*}
for every $s \in S$. In the general case we can find a diagonal matrix $D$ as well as some orthogonal matrix $U$ such that $Q'=UDU^{*}$. In view of $\skp{Q' f(s)^{*}u}{f(s)^{*}u}=\skp{Dg(s)^{*}u}{g(s)^{*}u}$ for $g(s):=f(s)U$ we can argue as before that $\int_S \norm{g(s)}^{2} \, \Theta(ds)<\infty$. As the following inequality holds for suitable constants $C_1,C_2>0$, where $\norm{\cdot}_F$ denotes the Frobenius norm, this gives the assertion.
\begin{equation*}
\norm{g(s)}=\norm{g(s)^{*}} \ge C_1 \norm{U^{*} f(s)^{*}}_F= C_1 \norm{f(s)^{*}}_F \ge C_2 \norm{f(s)}, \quad s \in S.
\end{equation*}
For the case $\alpha<2$ it is clear that $f \in \mathcal{I}(M)$ is equivalent to the finiteness of $L_f$, admitting the following representation due to \eqref{eq:23051701}.
\begin{align*}
& \integral{S}{}{\integral{\R^m}{}{\min \{1,\norm{f(s)x}^2\} }{\nu'(ds)}  }{\Theta(ds)} \\
&= \integral{S}{}{\integral{S_{\alpha}}{}{\integral{0}{\infty}{\min\{1,r^{\frac{2}{\alpha}} \norm{f(s) \zeta }^2  \}} r^{-2} }{\sigma(d \zeta )}}{\Theta(ds)} \\
&= \integral{S}{}{\integral{S_{\alpha}}{}{\indikator{(0,\infty)}{\norm{f(s) \zeta }} \left( \integral{0}{\norm{f(s) \zeta }^{-\alpha}}{r^{\frac{2}{\alpha}-2}\norm{f(s) \zeta }^2}{dr}  + \integral{\norm{f(s)\zeta }^{-\alpha}}{\infty }{r^{-2} }{dr} \right) }{\sigma(d\zeta )} }{\Theta(ds)} \\
&= \frac{2}{2-\alpha} \integral{S}{}{\integral{S_{\alpha}}{}{\norm{f(s)\zeta }^{\alpha} }{\sigma(d\zeta )} }{\Theta(ds)}.
\end{align*}
This implies (b). The case $\psi(\cdot)=-c \norm{\cdot}^{\alpha}$ can be treated similar as in part (a) since $f \in \mathcal{I}(M)$ implies that $s \mapsto \norm{f(s)^{*}u}^{\alpha}$ is integrable with respect to $\Theta$ for every $u \in \R^m$, see \cref{24051701} (iii).
\end{proof}
As announced before, we subsequently want to formulate sufficient conditions for $f \in \mathcal{I}(M)$, given that the generator $\mu$ is strictly operator-stable (and still full). In light of the previous example, which we therefore extend in two ways (see \cref{24051703} (ii)), these conditions will be sharp in general. 
\begin{theorem} \label{24061707}
Let $f:S \rightarrow \Li{m}$ be measurable and assume that $\mu$ is strictly operator-stable with exponent $B$. Furthermore, we suppose that there are $0< \delta_1 \le \Lambda_B^{-1}$ and $\delta_2>0$ such that 
\begin{equation*}
\integral{\{s:\norm{f(s)} \le R\} }{}{\norm{f(s)}^{\frac{1}{\Lambda_B}-\delta_1} }{\Theta(ds)} + \integral{\{s:\norm{f(s)} > R\} }{}{\norm{f(s)}^{\frac{1}{\lambda_B}+\delta_2} }{\Theta(ds)}< \infty
\end{equation*}
holds for some $R>0$. Then it follows that $f \in \mathcal{I}(M)$. Additionally, if $B$ is symmetric, we can choose $\delta_1=0$ and $\delta_2=0$, respectively.
\end{theorem}
\begin{proof}
Under the given assumptions we fix $u \in \R^m$ and define the sets
\begin{align*}
A_0:&=\{s: \norm{f(s)} \le R \text{ and } f(s)^{*}u \ne 0\}, \\
A_1:&=\{s: \norm{f(s)} > R \text{ and } 0<\norm{f(s)^{*}u} \le1\}, \\
A_2:&=\{s: \norm{f(s)}> R \text{ and } \norm{f(s)^{*}u} > 1 \}.
\end{align*}
Let $(\tau(\cdot),l(\cdot))$ denote the polar coordinates w.r.t $B^{*}$. Since $S_{B^{*}}$ is compact we can assume for convenience that $|\psi(\cdot)|\le 1$ on $S_{B^{*}}$. Hence the following calculation works for every $s \in S$ due to $\psi(0)=0$, \eqref{eq:18051711} and \eqref{eq:18051701}-\eqref{eq:18051702}:
\allowdisplaybreaks
\begin{align*}
&|\psi(f(s)^{*}u)| \\
&= |\psi(\tau(f(s)^{*}u)^{B^{*}} l(f(s)^{*}u) )| (\indikator{A_0}{s} +\indikator{A_1}{s} +\indikator{A_2}{s} ) \\
& \le C_0 \norm{f(s)^{*}u}^{1/ \Lambda_{B^{*}}- \delta_1} \indikator{A_0}{s} +C_1 \norm{f(s)^{*}u}^{1/ \Lambda_{B^{*}}- \delta_1} \indikator{A_1}{s}+C_2 \norm{f(s)^{*}u}^{1/ \lambda_{B^{*}}+ \delta_2} \indikator{A_2}{s} \\
& \le C_0 (\norm{u} \, \norm{f(s)})^{1/ \Lambda_{B}- \delta_1} \indikator{A_0}{s} +C_1(R \norm{u})^{1/ \Lambda_B-\delta_1} (R^{-1} \norm{f(s)})^{1/ \lambda_{B}+ \delta_2} \indikator{A_1}{s} \\
& \qquad \qquad +C_2 (\norm{u} \, \norm{f(s)})^{1/ \lambda_{B}+ \delta_2} \indikator{A_2}{s} \\
& \le C_0 (\norm{u} \,  \norm{f(s)})^{1/ \Lambda_{B}- \delta_1} \indikator{\{s:\norm{f(s)}\le R \} }{s} \\
& \qquad \qquad + [C_1(R \norm{u})^{1/ \Lambda_B-\delta_1} R^{-(1/ \lambda_B+\delta_2)}+ C_2  \norm{u}^{1/ \lambda_B+\delta_2} ] \norm{f(s)}^{1/ \lambda_B+\delta_2} \indikator{\{s:\norm{f(s)} > R \} }{s}.
\end{align*}
Note that we also used $\norm{f(s)}=\norm{f(s)^{*}}$ as well as $\text{spec}(B)=\text{spec}(B^{*})$. Furthermore, the constants $C_0,C_1,C_2>0$ come from \eqref{eq:18051701}-\eqref{eq:18051702}, where only $C_0$ depends on $R$ and $u$. More precisely, $C_0$ can be chosen monotone in $R \norm{u}$ (see the proof of Lemma 2.1 in \cite{bms}). Overall the first condition of \cref{24051701} (iii) is fulfilled. For the second one consider an arbitrary sequence $(u_n)\subset \R^m$ with limit $u$ and assume that this sequence is bounded by some $K>0$. Then Lemma 8.6 in \cite{Sato} and the Cauchy-Schwarz inequality yield that
\begin{equation*}
|\cos \skp{f(s)^{*}u_n}{x}-1 | \le 2 \max \{1,\norm{f(s)}^2 K^2\} \min \{1,\norm{x}^2\}, \quad n \in \N.
\end{equation*}
Hence we get by dominated convergence that
\begin{equation} \label{eq:30051701}
\integral{\R^m}{}{(\cos \skp{f(s)^{*}u_n}{x} -1)}{\nu' (dx)} \konv{n}{\infty} \integral{\R^m}{}{(\cos \skp{f(s)^{*}u}{x} -1)}{\nu'(dx)}, \quad s \in S.
\end{equation}
Then the same argument gives the assertion as we have
\allowdisplaybreaks
\begin{align*}
\left |\, \integral{\R^m}{}{(\cos \skp{f(s)^{*}u_n}{x}-1)}{\nu'(dx)}  \right|  & \le \frac{1}{2} \skp{Q' f(s)^{*}u_n}{f(s)^{*}u_n} + \integral{\R^m}{}{(1-\cos \skp{f(s)^{*}u_n}{x})}{\nu'(dx)} 
\end{align*}
and since the last term is bounded by $|\psi(f(s)^{*}u_n)|$ which can be treated as before again. For the additional statement merely observe that we can find an orthogonal matrix $O$ as well as a diagonal matrix $A=\text{diag}(a_1,...,a_m)$ sucht that $\norm{s^B}=\norm{O s^{A} O^{-1}} \le \norm{O} \norm{O^{-1}} \norm{s^A}$ according to Proposition 2.2.2 (e) in \cite{thebook}, whereas the estimates in \eqref{eq:17051701} then obviously hold with $\delta=0$ because of $s^A=\text{diag}(s^{a_1},...,s^{a_m})$. This can be used to improve \eqref{eq:18051701}-\eqref{eq:18051702} accordingly.
\end{proof}
Obviously the previous theorem implies the sufficient conditions of \cref{30051702}. Finally we want to examine further properties of the stochastic integral when using an (strictly) operator-stable generator. However, part (a) of the following statement is similar to Remark 2.1 in \cite{LiXiao} and also holds for a general infinitely-divisible generator, as long as we maintain the fullness of $\mu$.
\begin{prop} \label{20061701} 
Let $\mu$ be as in the previous theorem. Then we have for $f \in \mathcal{I}(M)$:
\begin{itemize}
\item[(a)] If there exists a set $A \in \Sigma$ with $\Theta'(A)>0$ such that $f(s) \in$ GL$(\R^m)$ for all $s \in A$, then the stochastic integral $I_M(f)$ is full.
\item[(b)] If $f(s) B =B f(s)$ holds $\Theta$-a.e., then $I_M(f)$ is also (strictly) operator-stable with exponent $B$. More precisely, the shifts of $I_M(f)$ can be obtained by \eqref{eq:30051706} below.
\end{itemize}
\end{prop}
\begin{proof}
\begin{itemize}
\item[(a)] Suppose that $I_M(f)$ is not full. Then Lemma 1.3.11 in \cite{thebook}, Example 3.7 (a) and Theorem 5.4 in \cite{integral} provide the existence of an $u \in \Gamma_m $ such that $\int_S \text{Re } \psi(f(s)^{*}au) \Theta(ds)=0$ for every $a \in \R$. Since $\text{Re } \psi (\cdot) \le 0$ we obtain $\Theta$-null sets $E_a$ such that $\text{Re } \psi (a f(s)^{*}u)=0$ for every $s \in E_A^c$. By continuity of $\psi$ and as $\Q$ is dense in $\R$ this implies the existence of a further $\Theta$-null set $E$ with $\text{Re } \psi (a f(s)^{*}u)=0$ for all $s \in E^c$ and $a \in \R$. Let $\tilde{u}:= f(\tilde{s})u \ne 0$ for some $ \tilde{s} \in A \cap E^c \ne \emptyset$ to see that $\mu$ is not full according to Lemma 1.3.11 in \cite{thebook}, which yields a contradiction.
\item[(b)] The given assumption ensures that $r^B f(s)=f(s) r^B$ for every $r>0$ and $ s \in S$ (except for a potential null set which we neglect subsequently), such that \eqref{eq:18051711} implies
\begin{equation*}
r \cdot \psi(f(s)^{*}u)=\psi((f(s)r^B)^{*}u)+\im \skp{a_r}{f(s)^{*}u} =\psi(f(s)^{*}r^{B^{*}}u)+\im \skp{a_r}{f(s)^{*}u}
\end{equation*}
for any $r>0, u \in \R^m$ and $s \in S$. Hence and with respect to the remarks at the beginning of chapter 2 as well as Theorem 5.4 in \cite{integral} again we obtain
\allowdisplaybreaks
\begin{align}
\bigerwartung{\text{e}^{\im \skp{I_M(f)}{u}}}^r &= \exp \left(\, \integral{S}{}{\left[ \psi(f(s)^{*}r^{B^{*}}u)+\im \skp{a_r}{f(s)^{*}u} \right] }{\Theta(ds)} \right) \notag \\ 
&= \exp \left( \, \integral{S}{}{\psi(f(s)^{*}r^{B^{*}}u)}{\Theta(ds)} \right)  \cdot  \exp \left( \im \integral{S}{}{\skp{f(s)a_r}{u} }{\Theta(ds)} \right) \notag  \\
&= \bigerwartung{\text{e}^{\im \skp{I_M(f)}{r^{B^{*}}u}}} \cdot \exp \left( \im \skpzwei{\integral{S}{}{f(s)a_r}{\Theta(ds)}}{u} \right), \label{eq:30051706} 
\end{align}
since $f \in \mathcal{I}(M)$ implies that the second step is valid for every $u \in \R^m$. Therefore the last one, too (consider unit vectors for $u$).
\end{itemize}
\end{proof}
Part (b) of \cref{20061701} reveals a new condition (obsolete for $\alpha$-stable generators) which seems to be quite challenging. However, the latter aspect can be softened by considering integrands $f$ which have an exponential form. Also note that the previous proposition can be extended similarly for $\C^m$-valued random measures with an $\R^{2m}$-valued generator due to the remarks in \cite{integral} and by means of the so called associated mapping of $f$ (see chapter 5 below). In contrast, the \textit{partial perspective}, which has also been presented there, appears more mentionable and requires a specific form of the exponent (due to the different dimensions that occur). In any case we recall some of the results in \cite{integral} needed later.
\begin{cor} \label{29061701}
Replace $\mu$ by an $\R^{2m}$-valued full and (strictly) operator-stable distribution $\tilde{\mu}$ as generator; denote the resulting $\R^{2m}$-valued ISRM by $M$ and the identified $\C^m$-valued one by $\widetilde{M}$, i.e. $M$ is the real associated ISRM of $\widetilde{M}$ in the sense of \cite{integral}. Then we have for any $f \in \mathcal{I}_p(\widetilde{M})=\{f: \text{$f$ is partially integrable w.r.t $\widetilde{M}$}\}$:
\begin{itemize}
\item[(a)] If there exists a set $A \in \Sigma$ with $\Theta'(A)>0$ such that $|\text{det}(\text{Re } f(s))| + |\text{det}(\text{Im } f(s))|>0$ for all $s \in A$, then $\text{Re } I_{\widetilde{M}}(f)$ is full.
\item[(b)] Assume that $(\text{Re }f(s))  B = B (\text{Re }f(s))$ and $(\text{Im }f(s))  B = B (\text{Im }f(s))$ hold $\Theta$-a.e., where $\widetilde{B} \in$ L$(\R^{2m})$ is an exponent of $\widetilde{\mu}$ which admits the following representation:
\begin{equation*}
\widetilde{B}=\begin{pmatrix}B & 0 \\ 0 & B \end{pmatrix} =:B \oplus B \quad \text{for some $B \in$ L$(\R^m)$.}
\end{equation*}
Then $\text{Re } I_{\widetilde{M}}(f)$ is also (strictly) operator-stable with exponent $B$. More precisely, the shifts of $\text{Re } I_{\widetilde{M}}(f)$ are given by 
\begin{center}
$\integral{S}{}{ [  (\text{Re }f(s)) a_{r,1}- (\text{Im } f(s)) a_{r,2} ] }{ \Theta(ds)}, \quad r>0$,
\end{center}
where $a_r=(a_{r,1},a_{r,2})$ are the shifts of $\widetilde{\mu}$ for $r>0$.
\end{itemize}
\end{cor}
\begin{proof}
In view of Remark 5.12 in \cite{integral} and the fact that $r^{(B \oplus B)^{*}}=r^{B^{*}} \oplus r^{B^{*}}$ for every $r>0$, the proof can be deduced easily from \cref{20061701}.
\end{proof}
\section{Generalized domains of attraction}
From now on we want to consider multivariate random fields of the form $\mathbb{X}=\{X(t):t \in \R^d\}$ with $X(t)$ being $\R^m$-valued. In this context we are interested in the modelilng of dependence structures, more precisely we want to focus on operator-self-similar random fields as introduced in \cite{LiXiao}. In fact, the following definition coincides mostly with the corresponding one in \cite{LiXiao} and extends the idea of \cite{Sato}. Generally, both sources provide plenty of further properties, relations and applications that underline the importance of this concept.
\begin{defi} \label{31051703}
For $E \in Q(\R^d)$ we call $\mathbb{X}$ \textit{wide-sense operator-self-similar} (short: WOSS) with \textit{time-scaling exponent} $E$, if for every $r>0$ there exist $B_r  \in$ L$(\R^m)$ and a \textit{shift function} $b_r: \R^d \rightarrow \R^m$ such that
\begin{equation} \label{eq:31051701}
\{X(r^E t): t \in \R^d\} \overset{fdd}{=} \{B_r X(t)+b_r(t): t \in \R^d\}.
\end{equation}
Furthermore, we say that $\mathbb{X}$ is \textit{operator-self-similar} (short: OSS), if the shift functions $b_r$ are constant for every $r>0$ or \textit{strictly operator-self-similar}, if they even fulfill $b_r \equiv 0$, respectively.
\end{defi}
Observe that neither the time-scaling exponent $E$ nor the operators $(B_r)$ are unique. Additionally we call $\mathbb{X}$ \textit{full} if this applies to the distribution of $X(t)$ for each $t \ne 0$. In this case \eqref{eq:31051701} and Lemma 2.3.5 in \cite{thebook} imply that $(B_r)_{r>0} \subset$ GL$(\R^m)$. This can be precised by the following theorem which is essentially due to Theorems 2.1 and 2.2 as well as Corollary 2.1 in \cite{LiXiao}. 
\begin{theorem}
Let $\mathbb{X}=\{X(t):t \in \R^d\}$ be a full, stochastically continuous random field with values in $\R^m$ and consider $E \in Q(\R^d)$ arbitrary. 
\begin{itemize}
\item[(a)] If $\mathbb{X}$ is WOSS with time-scaling exponent $E$, then there exist an operator $D \in \Li{m}$ with $\lambda_D \ge 0$ and a continuous function $b_r(t):(0,\infty) \times \R^d \rightarrow \R^m$ such that
\begin{equation} \label{eq:31051705}
\forall r >0: \qquad \{X(r^E t): t \in \R^d\} \overset{fdd}{=} \{r^D X(t)+b_r(t):t \in \R^d\}.
\end{equation}
Furthermore we have: $X(0)=a$ a.s. for some $a \in \R^m$ if and only if $D \in Q(\R^m)$ and in this case $b_0 \equiv a$ leads to a continuous extension of $b_r(t)$.
\item[(b)] If $X$ is (strictly) $OSS$, part (a) holds similarly, where $b_r(t)$ can be chosen constant in $t$ for every $r>0$ (with $b_r(t) \equiv 0$ and $a=0$, respectively).
\end{itemize}
\end{theorem}
We call the operator $D$ (not unique again) sucht that \eqref{eq:31051705} holds a \textit{space-scaling exponent} of $\mathbb{X}$. On the other hand we call the random field $(E,D)$-OSS, if merely \eqref{eq:31051705} is fulfilled with $b_r(t)=0$, hence implying the strict case. Also note that every WOSS field that satisfies the assumptions of the previous theorem becomes an $(E,D)$-OSS field by a deterministic transformation (see \cite{LiXiao} again). \\
We turn over to the main aspect of this section which underlines the theoretical importance of operator-self-similar random fields. The following definition seems to be a natural extension of the corresponding univariate notation (see section 11 in \cite{thebook} for example) and at the same time it allows us to prove the desired result in a very comprehensive way. This may fail as soon as we would allow general operators $T_s$ instead of $s^V$ subsequently, as the definition of operator-self-similarity is quite restrictive already concerning the scaling of time.
\begin{defi} \label{07111701}
Let $\mathbb{X}=\{X(t):t \in \R^d\}$ and $\mathbb{Y}=\{Y(t):t \in \R^d\}$ be random fields with values in $\R^m$, respectively. Then we say that $\mathbb{Y}$ belongs to the \textit{generalized domain of attraction of $\mathbb{X}$} (short: GDOA$(\mathbb{X})$), if there exist an operator $V \in Q(\R^d)$ as well as operators $A_s \in \Li{m}$ and functions $a_s:\R^d \rightarrow \R^m$ for every $s>0$ such that
\begin{equation} \label{eq:31051710}
\{A_s Y(s^V t)+a_s(t): t \in \R^d\} \overset{fdd}{\Longrightarrow} \{X(t):t \in \R^d\} \quad \text{as $s \rightarrow \infty$}.
\end{equation}
Here $\overset{fdd}{\Rightarrow}$ means convergence of all finite-dimensional distributions. Moreover, if the function $a_s$ can be chosen constant (with $a_s \equiv 0$) for every $s>0$, we say that $\mathbb{Y}$ belongs to the (\textit{strict}) \textit{domain of attraction of $\mathbb{X}$} (short: DOA$(\mathbb{X})$ and DOA${}_s(\mathbb{X})$, respectively).
\end{defi}
Let us state the main result of this section which extends the corresponding result in \cite{HuMa82}.
\begin{theorem} \label{31051720}
Let $\mathbb{X}$ be as in \cref{07111701}. Assume additionally that $\mathbb{X}$ is full and stochastically continuous. Then we have: $\mathbb{X}$ is WOSS if and only if GDOA$(\mathbb{X})\ne \emptyset$. Moreover, every operator $V$ that fulfills \eqref{eq:31051710} is an time-scaling exponent of $\mathbb{X}$ and vice versa.
\end{theorem}
We need two auxiliary results which might be well-known. They follow by a slight refinement of the proofs of Lemma 2.3.8, 2.3.9 and 2.3.16 in \cite{thebook}. The generalized versions that we present now will be useful in \cref{31051721} below. 
\begin{lemma} \label{08061701}
Let $\nu,\mu, \mu_1,...$ be probability measures on $\R^m$ and $(a_n) \subset \R^m$ a sequence of vectors as well as $(A_n) \subset \Li{m}$ a sequence of operators.
\begin{itemize}
\item[(a)] If $\mu_n \rightarrow \mu$ and $\mu_n * \eps_{a_n} \rightarrow \nu$ weakly, then the sequence $(\norm{a_n})$ is bounded.
\item[(b)] Suppose additionally that $\mu$ is full. If $\mu_n \rightarrow \mu$ and $(A_n \mu_n)* \eps_{a_n} \rightarrow \nu$ weakly, then both sequences $(\norm{A_n})$ and $(\norm{a_n})$ are bounded.
\end{itemize}
\end{lemma}
\begin{proof}[Proof of \cref{31051720}]
Assume that \eqref{eq:31051701} holds. Then $\mathbb{X} \in$ GDOA$(\mathbb{X})$, i.e. GDOA$(\mathbb{X}) \ne \emptyset$, can be verified with $\mathbb{Y}:=\mathbb{X}, A_s:=B_s^{-1},a_s(t):=-B_S^{-1}b_s(t)$ and $V:=E$ in \eqref{eq:31051710}. Note that $B_s \in$ GL$(\R^m)$ for every $s>0$ by fullness of $\mathbb{X}$. Conversely, \eqref{eq:31051710} implies the following convergence in distribution for $k \in \N$ and $t_1,...,t_k \in \R^d$ arbitrary:
\begin{equation} \label{eq:07061701}
(A_s Y(s^V t_1)+a_s(t_1),...,A_s Y(s^V t_k)+a_s(t_k)) \konviv (X(t_1),...,X(t_k)) \quad (\text{as } s \rightarrow \infty).
\end{equation}
For the rest of this proof we fix some $r>0$ and $t \in \R^d$. Then \eqref{eq:07061701} implies on one hand that $A_{nr} Y((nr)^{V}t)+a_{nr}(t) \Rightarrow X(t)$ and on the other hand that $A_n Y((nr)^{V}t)+a_n(r^V t) \Rightarrow X(r^V t)$ as $n \rightarrow \infty$. Here we can temporary assume $t \ne 0$ such that $A_{nr}$ (independent from $t$) has to be invertible due to Lemma 2.3.7 in \cite{thebook} for large $n$. Hence we can define $H_n:=A_n A_{nr}^{-1}$ as well as $h_n(t):=a_n(r^V t)-A_n A_{nr}^{-1}a_{nr}(t)$ to observe that we also have the convergence of $H_n (A_{nr} Y((nr)^V t)+a_{nr}(t))+h_n(t)$ as $n \rightarrow \infty$, again with limit $X(r^V t)$. Now combine the previous findings to see that the sequences $(\norm{H_n})$ and $(\norm{h_n(t)})$ have to be bounded in the light of \cref{08061701} (b). Thus, without loss of generality, we can assume that $H_n \rightarrow H$ for some $H \in \Li{m}$. Similarly, by writing $\Q^d=\{q_1,q_2,...\}$, we see that $(h_n(q_1))$ contains a suitable subsequence $(h_{n_{l,1}}(q_1))_l$ that converges, say with limit $h(q_1)$. However $(h_{n_{l,1}}(q_2))_l$ is bounded, too, with a converging subsequence $(h_{n_{l,2}}(q_2))_l$ and limit $h(q_2)$. Proceed inductively. Then, for given $k \in \N$ and $t_1,...,t_k \in \Q^d$, we can find some $k'$ with $\{t_1,...,t_k\} \subset \{q_1,...,q_{k'}\}$ and extend the thoughts above based on \eqref{eq:07061701}, namely 
\begin{equation*}
(A_{nr} Y((nr)^V q_1)+a_{nr}(q_1),...,A_{nr} Y((nr)^V q_{k'})+a_{nr}(q_{k'})) \konviv (X(q_1),...,X(q_{k'})) 
\end{equation*}
as well as (where we denote the random vector on the left-hand side by $Z$ meanwhile)
\begin{equation*}
\text{diag}(H_n,...,H_n) Z + (h_n(q_1),...,h_n(q_{k'})) \konviv (X(r^Vq_1),...,X(r^V q_{k'})).
\end{equation*}
Since $H_n \rightarrow H$, we can consider the subsequence $(n_{l,k'})$ and Theorem 2.1.8 in \cite{thebook} yields 
\begin{equation} \label{eq:08061702}
\{X(r^V t):t \in \Q^d \} \overset{fdd}{=} \{HX(t)+h(t): t \in \Q^d\}
\end{equation}
after projection. Furthermore, for any $t \in \R^d \setminus \Q^d$ there is a sequence $(q_n')\subset \Q^d$ with limit $t$ while we get $X(r^V q_n') \Longrightarrow X(r^V t)$ by assumption. Hence \cref{08061701} (a) implies that $(\norm{h(q_n')})$ is bounded due to \eqref{eq:08061702} and since $H X(q_n') \Rightarrow H X(t)$ as before. Thus we can choose - and memorize - a further subsequence $(h(q_{n_l}'))$ with some limit $h(t)$ again. Overall this gives a mapping $h:\R^d \rightarrow \R^m$ such that for any $t \in \R^d$ there is a (renamed) sequence $(q_{n,t})\subset \Q^d$ with $q_{n,t} \rightarrow t$ and $h(q_{n,t}) \rightarrow h(t)$ as $n \rightarrow \infty$. Finally, for $k \in \N$ and $t_1,...,t_k \in \R^d$ arbitrary, we see by assumption that 
\begin{equation} \label{eq:08061703}
(X(r^V q_{n,t_1}),...,X(r^V q_{n,t_k})) \konviv (X(r^V t_1),...,X(r^V t_k)).
\end{equation}
At the same time and similar as before we observe with \eqref{eq:08061702} that
\begin{equation} \label{eq:08061704}
(X(r^V q_{n,t_1}),...,X(r^V q_{n,t_k})) \konviv (H X( t_1)+h(t_1),...,H X(t_k)+h(t_k)).
\end{equation}
This extends \eqref{eq:08061702} for $t \in \R^d$, if we combine \eqref{eq:08061703} with \eqref{eq:08061704} and therefore gives the assertion since $r>0$ was arbitrary.
\end{proof}
\begin{remark} \label{31051721}
The assumptions in \cref{31051720} can be relaxed twice. First it is enough to claim that $X(t')$ is full for some $t' \in \R^d$: Use $X(t')=X(r^E r^{-E}t')$ together with Lemma 2.3.5 in \cite{thebook} for necessity. Conversely, \cref{08061701} (a) works again. \\
Furthermore, the proof showed that continuity in distribution is enough as long as it holds for all finite-dimensional random vectors as in \eqref{eq:08061703}. 
\end{remark}
The first part of the previous remark also holds for the next statement.
\begin{cor}
Let $\mathbb{X}=\{X(t):t \in \R^d\}$ be a full random field with values in $\R^m$. Then we have: $\mathbb{X}$ is (strictly) OSS if and only if DOA${}_{(s)}(\mathbb{X}) \ne \emptyset$. Moreover, every operator $V$ that fulfills \eqref{eq:31051710} is an time-scaling exponent of $\mathbb{X}$ and vice versa.
\end{cor}
\begin{proof}
Suppose $\mathbb{X}$ to be (strictly) OSS, then $a_s=-B_s^{-1}b_s$ is independent from $t$ for every $s>0$. Conversely, if $a_s(\cdot) =a_s$ is constant for every $s>0$, the same is true for $h_n$. This simplifies the proof of \cref{31051720} since we no longer need to consider $\Q^d$.
\end{proof}
\section{Moving-average representation} \label{chapter4}
First of all and inspired by the $\alpha$-stable case considered in \cite{SaTaq94} we define:
\begin{defi}
An $\R^m$-valued random field $\mathbb{X}=\{X(t):t \in \R^d\}$ is called (symmetric or strictly) operator-stable, if all finite-dimensional distributions of $\mathbb{X}$ are (symmetric or strictly) operator-stable. Additionally, we say that $B \in$ L$(\R^m)$ is an \textit{exponent} of $\mathbb{X}$, if the linear combinations of the form $\sum_{j=1}^{k} c_j X(t_j)$ are operator-stable with exponent $B$ for any finite choices $c_1,...,c_k \in \R$ and $t_1,...,t_k \in \R^d$.
\end{defi}
Observe that in difference to the $\alpha$-stable case, in the more general operator-stable situation an exponent $B$ of $\mathbb{X}$ may not exist in general. \\
We now want to extend \eqref{eq:16051701} and the multivariate counterpart in \cite{LiXiao} to the more general case of strictly operator-stable marginals. In order to do so we will use an operator-stable ISRM $M$ and a suitable family $(f_t)_{t \in \R^d} \subset \mathcal{I}(M)$ such that $X(t)=\int_{\R^d} f_t(s) \, M(ds)$. \\
More precisely, let $\mu$ be a full, strictly operator-stable distribution on $\R^m$ with log-characteristic function $\psi$ and exponent $B$ that generates $M$ together with the $d$-dimensional Lebesgue measure on $(S,\Sigma)=(\R^d,\B(\R^d))$. An appropriate choice of integrands will be ensured by two types of deterministic functions: According to \cite{bms} a function $\phi:\R^d \rightarrow \C$ is called \textit{$E$-homogeneous} for some $E \in Q(\R^d)$, if $\phi(c^E x)=c \, \phi(x)$ for all $c>0$ and $x \in \Gamma_d=\R^d \setminus \{0\}$. Moreover, given such an operator $E$ and some $\beta>0$, a function $\phi:\R^d \rightarrow [0,\infty)$ being strictly positive outside the origin is called \textit{$(\beta,E)$-admissible}, if for any $0<A<B$ there exists a constant $C>0$ such that the following implication holds for $A \le \norm{z} \le B$:
\begin{equation*}
\tau_E(x) \le 1 \quad \Rightarrow \quad |\phi(x+z)-\phi(z)| \le C \tau_E(x)^{\beta},
\end{equation*}
where we recall that $x=\tau_E(x)^{E} l_E(x)$ for any $x \ne 0$. See \cite{bms} for more details on $(\beta,E)$-admissible functions and examples. 
\begin{theorem} [Moving-average representation] \label{14061705}
Assume that $D \in Q(\R^m)$ and $E \in Q(\R^d)$ with $q:= \text{trace}(E)$. Furthermore, let $\phi$ be a continuous, $E$-homogeneous and $(\beta,E)$-admissible function for some $\beta>0$ such that $\lambda_{D-qB}+\lambda_{qB}>0$ and $\Lambda_{D-qB} + \Lambda_{qB} < \beta$. Then the following stochastic integral exists for each $t \in \R^d$:
\begin{equation} \label{eq:05061701}
X_{\phi,D}(t):= \integral{\R^d}{}{\left[ \phi(t-s)^{D-qB}- \phi(-s)^{D-qB} \right]}{M(ds)}.
\end{equation}
We call the resulting random field $\mathbb{X}_{\phi,D}:=\{X_{\phi,D}(t):t \in \R^d\}$ an $(\phi,D)$-moving-average representation (with respect to $M$).
\end{theorem}
\begin{proof}
Throughout we define $0^A:=0 \in \Li{m}$ and recall that Lebesgue null-sets are negligible due to Theorem 3.2 in \cite{integral}. We have 
\begin{equation*}
\zeta_1:=\frac{\lambda_{D-qB}+\lambda_{qB}}{\Lambda_B}>0, \quad \zeta_2:=\frac{\Lambda_{D-qB}+\Lambda_{qB}-\beta }{\Lambda_B}<0
\end{equation*}
and moreover we can choose some
\begin{equation*}
0< \delta < \min \left \{ \frac{1}{\Lambda_B}, \frac{\zeta_1}{2 |\lambda_{D-qB}|}, \frac{|\zeta_2|}{2 |\Lambda_{D-qB}-\beta|} \right \}.
\end{equation*}
If $\lambda_{D-qB}=0$ or $\Lambda_{D-qB}=\beta$, the corresponding terms are left out of the minimum. Finally choose $\eps_1,\eps_2>0$ such that
\begin{equation*}
\eps_1< \frac{\zeta_1}{2 (1/ \lambda_B+\delta)}, \quad \eps_2< \min \left \{ \frac{|\zeta_2|}{2(1/ \lambda_B+\delta)}, \Lambda_{qB} \right \}.
\end{equation*}
Obviously we have $X_{\phi,D}(0)=0$ a.s. In view of \cref{24061707} it suffices to show that 
\begin{equation} \label{eq:13061711}
\integral{\R^d}{}{ \left[ \norm{\phi(t-s)^{D-qB}- \phi(-s)^{D-qB}}^{\frac{1}{\Lambda_B}-\delta} +\norm{\phi(t-s)^{D-qB}- \phi(-s)^{D-qB}}^{\frac{1}{\lambda_B}+\delta} \right]}{ds} < \infty
\end{equation}
for any fixed $t \in \Gamma_d$. The occurring constants are denoted by $C_1,C_2,...$ and we write $(\tau(\cdot),l(\cdot)))$ for the polar coordinates with respect to $E$ with $\tau(0):=0$. As in the proof of Theorem 2.5 in \cite{LiXiao} we see that 
\begin{equation} \label{eq:13061705}
m_{\phi} \tau(s) \le \phi(s) \le M_{\phi} \tau(s) \quad \forall s \in \R^d,
\end{equation}
where $m_{\phi}:=\min_{\theta \in S_E} \phi(\theta)>0$ and $M_{\phi}:=\max_{\theta \in S_E} \phi(\theta)>0$. Hence we get  
\begin{equation} \label{eq:13061701}
\norm{\phi(s)^{D-qB}} \le C_1 \norm{\tau(s)^{D-qB}} 
\end{equation}
for all $s \in \R^d$ and for some $C_1>0$. The same argument essentially yields
\begin{equation*}
M_1:= \underset{m_{\phi} \le r \le M_{\phi}}{\max} \norm{r^E} >0, \quad M_2:=\underset{1/M_{\phi} \le r \le 1/m_{\phi}}{\max} \norm{r^E} >0
\end{equation*}
as well as
\begin{equation*}
0 < \upsilon:=\underset{\theta \in S_E}{\min} \norm{\theta} \le \Upsilon := \underset{\theta \in S_E}{\max} \norm{\theta} < \infty.
\end{equation*}
According to \cite{LiXiao} again we observe on one hand that
\begin{equation} \label{eq:13061702}
0< A':= \frac{\upsilon}{M_1} \le \norm{\phi(s)^{-E} s} \le M_2 \Upsilon=:B' \quad \forall s \in \Gamma_d.
\end{equation}
On the other hand the following implication holds for every $x,z \in \R^d$ and some $C_2>0$:
\begin{equation} \label{eq:13061703}
(A' \le \norm{z} \le B' \text{ and } \tau(x) \le 1) \quad \Rightarrow \quad |\phi(x+z)-\phi(z)| \le C_2 \tau(x)^{\beta}.
\end{equation}
Also note that \eqref{eq:13061705} and $\tau(s)=\tau(-s)$ ensure the existence of a $\gamma=\gamma(t)>0$ such that
\begin{equation} \label{eq:13061706}
\phi(s)^{-1} \tau(t)<1, \quad C_2 \phi(-s)^{- \beta}\tau(t)^{\beta}< \frac{1}{2}, \quad  \phi(-s)>1
\end{equation}
remains true for any $s \in \R^d$ with $ \tau(s) > \gamma$, respectively. To show \eqref{eq:13061711} for any $\kappa>0$ we define the bounded set $A(\kappa):=\{s: \tau(s) \le \kappa\}$. Using \eqref{eq:17051701} there exists a constant $C_3=C_3(t)>0$ such that the following estimates hold in virtue of \eqref{eq:13061701} for every $\rho>0$. We also use Proposition 2.3 in \cite{bms} with a corresponding bounded measure $\sigma$ on $S_E$ as well as $\tau(r^E \theta)=r$ for $\theta \in S_E$.
\begin{equation*}
\integral{A(\gamma(t))}{}{\norm{\phi(-s)^{D-qB}}^{\rho}}{ds}  \le C_1^{\rho} C_3^{\rho} \integral{A(\gamma(t))}{}{\tau(s)^{\rho(\lambda_{D-qB}-\eps_1)}}{ds} = C_1^{\rho} C_3^{\rho} \sigma(S_E) \integral{0}{\gamma(t)}{r^{\rho(\lambda_{D-qB}-\eps_1)+q-1}}{dr}.
\end{equation*}
Note that the integral on the right hand side of the above equation is finite for $\rho_1:=1/\Lambda_B-\delta$ and $\rho_2:=1/\lambda_B+\delta$. In fact, using $\text{spec}(qB)=q \text{ spec}(B)$ we get
\allowdisplaybreaks
\begin{equation*}
\rho_1 (\lambda_{D-qB}-\eps_1) +q = \frac{\lambda_{D-qB}+\Lambda_{qB}}{\Lambda_B} - \delta \lambda_{D-qB}- \eps_1 (1/ \Lambda_B-\delta)  \ge \zeta_1 - \delta \lambda_{D-qB}- \eps_1 (1/ \lambda_B+\delta)  >0.
\end{equation*}
Similarly we see that
\begin{equation*}
\rho_2 (\lambda_{D-qB}-\eps_1)+q = \frac{\lambda_{D-qB}+\lambda_{qB}}{\lambda_B} + \delta \lambda_{D-qB}- \eps_1 (1/ \lambda_B+\delta)  \ge \zeta_1 + \delta \lambda_{D-qB}- \eps_1 (1/ \lambda_B+\delta)  >0.
\end{equation*}
Moreover, Lemma 2.2. in \cite{bms} yields a constant $C_4 \ge 1$ with $\tau(x+y) \le C_4 (\tau(x)+\tau(y))$. Hence we observe that $\{s:\tau(t+s) \le \gamma(t)\} \subset \{s:\tau(s) \le C_4(\gamma(t)+\tau(t))\}$. By a change of variable and with $\gamma(t)':=C_4(\gamma(t)+\tau(t))$ this immediately gives 
\begin{equation*}
\integral{A(\gamma(t))}{}{\norm{\phi(t-s)^{D-qB} }^{\rho_i} }{ds}  \le \integral{A(\gamma(t)')}{}{\norm{\phi(-s)^{D-qB} }^{\rho_i} }{ds} < \infty, \quad i=1,2.
\end{equation*}
Because of $(a+b)^{\rho_i} \le C_5 (a^{\rho_i}+b^{\rho_i})$ for $a,b \ge 0$ and some suitable $C_5>0$, this finally proves that the integral in \eqref{eq:13061711} is finite on $A(\gamma(t))$. \\
Next we can argue as in the proof of \cite{LiXiao}, Theorem 2.5 to see that 
\begin{equation} \label{eq:14061702}
\norm{u^{D-qB}-I_m} \le C_6 \norm{D-qB} |u-1| \quad \forall \,1/2 <u< 3/2
\end{equation}
with a suitable constant $C_6>0$. Moreover, we obtain
\begin{equation} \label{eq:14061701}
\norm{\phi(t-s)^{D-qB}-\phi(-s)^{D-qB}} \le  \norm{\phi(\phi(-s)^{-E}t-\phi(-s)^{-E}s)^{D-qB}  -I_m} \cdot \norm{\phi(-s)^{D-qB}}
\end{equation}
for any $s \in \Gamma_d$, since $\phi$ is $E$-homogeneous. Consider $z(s):=- \phi(-s)^{-E}s=\phi(-s)^{-E}(-s)$ and $x(s,t):= \phi(-s)^{-E}t$ to verify in view of \eqref{eq:13061706} and \eqref{eq:13061702} that $\phi(z(s))=1, \tau(x(s,t))<1$ as well as $\norm{z(s)} \in [A',B']$ for any $s \in A(\gamma(t))^c$. Thus we can apply \eqref{eq:13061703} to get
\begin{equation*}
|\phi(\phi(-s)^{-E}t-\phi(-s)^{-E}s)-1| \le C_2 \tau(\phi(-s)^{-E}t)^{\beta}= C_2 \tau(t)^{\beta} \phi(-s)^{- \beta} 
\end{equation*}
for all $s \in A(\gamma(t))^c$. Particularly, we have $1/2 < u(s,t):=\phi(\phi(-s)^{-E}t-\phi(-s)^{-E}s) < 3/2$. Together with \eqref{eq:14061702} and \eqref{eq:14061701} this yields for $s \in A(\gamma(t))^c$ and any $\rho>0$ that
\allowdisplaybreaks
\begin{align*}
\norm{\phi(t-s)^{D-qB}-\phi(-s)^{D-qB}} ^{\rho} & \le C_6^{\rho} C_2^{\rho} \norm{D-qB}^{\rho} \tau(t)^{\rho \beta} \phi(-s)^{- \rho \beta}  \norm{\phi(-s)^{D-qB}}^{\rho} \\
& \le C_7^{\rho} C_6^{\rho} C_2^{\rho} \norm{D-qB}^{\rho} \tau(t)^{\rho \beta} \phi(-s)^{ \rho (\Lambda_{D-qB} + \eps_2-\beta)} 
\end{align*}
with some suitable $C_7>0$ based on \eqref{eq:17051701}. Observe that $\Lambda_{D-qB} + \eps_2-\beta<0$. Hence using \eqref{eq:13061705} and setting $C_8:=C_7 C_6 C_2 \norm{D-qB} m_{\phi}^{\Lambda_{D-qB}+\eps_2-\beta}$ we can compute that
\begin{equation*}
\integral{A(\gamma(t))^c}{}{\norm{\phi(t-s)^{D-qB}-\phi(-s)^{D-qB}} ^{\rho}}{ds} \le C_{8}^{\rho} \tau(t)^{\rho \beta} \sigma(S_E) \integral{\gamma(t)}{\infty}{r^{\rho (\Lambda_{D-qB} + \eps_2-\beta)+q-1}}{dr}
\end{equation*}
by using Proposition 2.3 in \cite{bms} again. Overall this shows that \eqref{eq:13061711} holds, since 
\allowdisplaybreaks
\begin{align*}
\rho_1(\Lambda_{D-qB}+\eps_2-\beta)+q & = \frac{\Lambda_{D-qB} + \Lambda_{qB}- \beta}{\Lambda_B} - \delta(\Lambda_{D-qB}-\beta)+ \eps_2 (1/ \Lambda_B- \delta)  \\
& \le \zeta_2 - \delta(\Lambda_{D-qB} - \beta)+ \eps_2 (1/ \lambda_B + \delta) \\
& <0
\end{align*}
as well as
\allowdisplaybreaks
\begin{align*}
\rho_2(\Lambda_{D-qB}+\eps_2-\beta)+q & = \frac{\Lambda_{D-qB} + \lambda_{qB}- \beta}{\lambda_B} + \delta(\Lambda_{D-qB}-\beta)+ \eps_2 (1/ \lambda_B+ \delta) \\
& \le  \frac{\Lambda_{D-qB} + \Lambda_{qB}- \beta}{\lambda_B} + \delta(\Lambda_{D-qB}-\beta)+ \eps_2 (1/ \lambda_B+ \delta) \\
& \le  \zeta_2 + \delta(\Lambda_{D-qB}-\beta)+ \eps_2 (1/ \lambda_B+ \delta) \\
& <0.
\end{align*}
\end{proof}
The following proposition provides some basic properties of the random field $\mathbb{X}_{\phi,D}$ constructed by means of \eqref{eq:05061701}. In fact, part (a) states that under an additional condition $\mathbb{X}_{\phi,D}$ is an example of an $(E,D)$-OSS random field that is also operator-stable with exponent $B$.
\begin{prop} \label{27061705}
Assume that the assumptions of \cref{14061705} are fulfilled. 
\begin{itemize}
\item[(a)] If $DB=BD$, then $\mathbb{X}_{\phi,D}$ is $(E,D)$-OSS and strictly operator-stable with exponent $B$.
\item[(b)] $\mathbb{X}_{\phi,D}$ is stochastically continuous and has stationary increments. 
\item[(c)] If $D-qB \in$ GL$(\R^m)$, then $\mathbb{X}_{\phi,D}$ is full.
\end{itemize}
\end{prop}
\begin{proof}
For part (a) let $k \in \N$ and $t_1,...,t_k \in \R^d$ be arbitrary. Then in view of Theorem 5.4 (b) in \cite{integral} the characteristic function of $(X_{\phi,D}(t_1),...,X_{\phi,D}(t_k))$ is given by
\begin{equation*}
\R^{k \cdot m} \ni u=(u_1,...,u_k) \mapsto \exp \left( \, \integral{\R^d}{}{ \psi \left( \sum_{j=1}^{k} (\phi(t_j-s)^{D-qB}-\phi(-s)^{D-qB})^{*}u_j \right) }{ds} \right).
\end{equation*}
Moreover we obtain by using \eqref{eq:18051711} that
\begin{equation*}
\psi \left( \sum_{j=1}^{k} (\phi(t_j-s)^{D-qB}-\phi(-s)^{D-qB})^{*}r^{B^{*}}u_j \right)= r \cdot \psi \left( \sum_{j=1}^{k} (\phi(t_j-s)^{D-qB}-\phi(-s)^{D-qB})^{*}u_j \right)
\end{equation*}
for any $r>0, u=(u_1,...,u_n) \in \R^{k \cdot m}$ and $s \in S$ since $B$ and $D-qB$ commutate under the given assumption. With $B_k:=B \oplus \cdots \oplus B$ and in view of $r^{B_k^{*}}u=(r^{B^{*}}u_1, ...,r^{B^{*}}u_k)$ this shows that the considered random vector is strictly operator-stable with exponent $B_k$. Using \cref{20061701} it then follows easily that $\mathbb{X}_{\phi,D}$ is an operator-stable random field with exponent $B$. Quite similar $DB=BD$, a change of variable and the $E$-homogeneity of $\phi$ yield 
\allowdisplaybreaks
\begin{align*}
& \bigerwartung{\text{e}^{\im \sum_{j=1}^{k} \skp{X_{\phi,D}(c^E t_j)}{u_j}} } \\
&= \exp \left( \, \integral{\R^d}{}{ \psi \left( \sum_{j=1}^{k} (c^{D-qB}( \phi(t_j-c^{-E}s)^{D-qB}-\phi(-c^{-E}s)^{D-qB}))^{*}u_j \right) }{ds} \right) \\
&= \exp \left( \, \integral{\R^d}{}{c^{q} \cdot \psi \left( \sum_{j=1}^{k} (c^{D-qB}( \phi(t_j-s)^{D-qB}-\phi(-s)^{D-qB}))^{*}u_j \right) }{ds} \right) \\
&= \exp \left( \, \integral{\R^d}{}{ \psi \left( \sum_{j=1}^{k} (c^{D-qB}( \phi(t_j-s)^{D-qB}-\phi(-s)^{D-qB})c^{qB})^{*}u_j \right) }{ds} \right) \\
&=  \bigerwartung{\text{e}^{\im \sum_{j=1}^{k} \skp{c^D X_{\phi,D}( t_j)}{u_j}} }
\end{align*}
for any $c>0$. Hence $\mathbb{X}_{\phi,D}$ is $(E,D)$-OSS. Concerning part (b) we first consider $t_1,...,t_k \in \R^d$ again as well as $h \in \R^d$ arbitrary. Define 
\begin{equation*}
Y_j:= \integral{\R^d}{}{[\phi(t_j+h-s)^{D-qB}-\phi(h-s)^{D-qB}]}{M(ds)}
\end{equation*}
and observe that $(Y_1,...,Y_k) \overset{d}{=} (X_{\phi,D}(t_1),...,X_{\phi,D}(t_k) )$ by change of variables and the use of characteristic functions as before. In light of $Y_j=X_{\phi,D}(t_j+h)-X_{\phi,D}(h)$, which holds a.s. by linearity, this gives the stationarity of increments since $X_{\phi,D}(0)=0$. Now fix some $t_0 \in \R^d$ and $u \in \R^m$ to observe that the stochastic continuity of $\mathbb{X}$ would follow in view of Theorem 5.4 (c) in \cite{integral} and after a change of variable, if we can show that
\begin{equation} \label{eq:20061702}
\integral{\R^d}{}{ \psi \left( (\phi(t-s)^{D-qB}-\phi(-s)^{D-qB})^{*}u\right) }{ds} \rightarrow 0 \quad (t \rightarrow 0).
\end{equation} 
Obviously the integrand converges to $\psi(0)=0$ almost everywhere by continuity. On the other hand we can argue similar as in the proof of \cref{24061707} to verify that there exists a constant $K>0$ (sufficiently large and only depending on $u$) such that the previous integrand is bounded by $K \sum_{j=1}^2 \norm{\phi(t-s)^{D-qB}-\phi(-s)^{D-qB}}^{\rho_j} $, where $\rho_1,\rho_2>0$ are chosen as in the proof of \cref{14061705} above. There we also proved that
\allowdisplaybreaks 
\begin{align}
& \norm{\phi(t-s)^{D-qB}-\phi(-s)^{D-qB}}^{\rho_j} \notag \\
& \le C_5 \left( C_1^{\rho_j} \tilde{C}_{3}^{\rho_j} \tau(t-s)^{\rho_j(\lambda_{D-qB}-\eps_1)} + C_1^{\rho_j} C_{3}^{\rho_j} \tau(s)^{\rho_j(\lambda_{D-qB}-\eps_1)}  \right) \indikator{A(\gamma(t))}{s} \notag \\
& \qquad + C_8^{\rho_j} \tau(s)^{\rho_j(\Lambda_{D-qB}+\eps_2-\beta)} \tau(t)^{\rho_j \beta } \indikator{A(\gamma(t))^c}{s} \label{eq:20061706}
\end{align}
for $j=1,2$. Here all constants are as in the proof of \cref{14061705}, except for $\tilde{C}_3=\tilde{C}_3(t)$ which we can choose such that $\norm{\tau(t-s)^{D-qB}} \le \tilde{C}_3 \tau(t-s)^{\lambda_{D-qB}-\eps_1}$ holds for all $s \in A(\gamma(t)^c)$ again. Then for any null sequence $(t_n)$ there exists some $n^{*}$ with $\tau(t_n)\le \tau(t_{n^{*}})$ for all $n \in \N$. Hence in view of \eqref{eq:13061706} we can assume without loss of generality that $\gamma(t_n)=\gamma(t_{n^{*}})$ for each $n \in \N$. Moreover, $C_3$ and $\tilde{C}_3$ were the only constants in \eqref{eq:20061706} depending on $t$. Since $\sup \{\tau(t_n-s): s \in A(\gamma(t_{n^{*}})), n \in \N \}<\infty$, we can assume that they are also independent of $t$. Overall there is a constant $C'>0$  (only depending on $u$) such that
\allowdisplaybreaks
\begin{align}
& \left| \psi \left( (\phi(t_n-s)^{D-qB}-\phi(-s)^{D-qB})^{*}u\right) \right| \notag \\
& \quad \le C' \summe{j=1}{2} \left[  \tau(t_n-s)^{\rho_j(\lambda_{D-qB}-\eps_1)}  \indikator{A(\gamma(t_{n^{*}}))}{s} \right. \notag \\
& \qquad \qquad +  \tau(s)^{\rho_j(\lambda_{D-qB}-\eps_1)}  \indikator{A(\gamma(t_{n^{*}}))}{s} \notag \\
& \left. \qquad \qquad +  \tau(s)^{\rho_j(\Lambda_{D-qB}+\eps_2-\beta)} \tau(t_n)^{\rho_j \beta} \indikator{A(\gamma(t_{n^{*}}))^c}{s} \right]  =: g_n(s) \notag
\end{align}
for any $s \in \R^d$. Because of $\tau(0)=0,\tau(s)=\tau(-s)$ and the continuity of $\tau(\cdot)$ we obtain
\begin{equation*}
g_n(s) \rightarrow 2 C' \summe{j=1}{2} \tau(-s)^{\rho_j(\lambda_{D-qB}-\eps_1)} \indikator{A(\gamma(t_{n^{*}}))}{s}=:g(s), \quad s \in \R^d.
\end{equation*} 
However, the proof of \cref{14061705} showed that the functions $g,g_1,...$ are integrable. In view of section 4, Theorem 19 in \cite{RoyFi} $\int_{\R^d} g_n(s) \, ds \rightarrow \int_{\R^d} g(s) \, ds$ would finally imply \eqref{eq:20061702}. Therefore it remains to show that
\allowdisplaybreaks
\begin{align}
\integral{\R^d}{}{ \tau(-s)^{\rho_j(\lambda_{D-qB}-\eps_1)} \indikator{A(\gamma(t_{n^{*}}))}{s} }{ds} &= \limes{n}{\infty} \integral{\R^d}{}{\tau(t_n-s)^{\rho_j(\lambda_{D-qB}-\eps_1)} \indikator{A(\gamma(t_{n^{*}}))}{s}}{ds} \label{eq:08111701} \\
&= \limes{n}{\infty} \integral{\R^d}{}{ \tau(-s)^{\rho_j(\lambda_{D-qB}-\eps_1)} \indikator{\{s: \tau(t_n+s) \le \gamma(t_{n^{*}})\} }{s}}{ds} \notag
\end{align}
for $j=1,2$. Observe that 
\begin{equation*}
\tau(-s)^{\rho_j(\lambda_{D-qB}-\eps_1)} \indikator{\{s: \tau(t_n+s) \le \gamma(t_{n^{*}})\} }{s} \rightarrow \tau(-s)^{\rho_j(\lambda_{D-qB}-\eps_1)} \indikator{A(\gamma(t_{n^{*}}))}{s}
\end{equation*}
as $n \rightarrow \infty$, at least except for the Lebesgue null set $\{s:\tau(s)=\gamma(t_{n^{*}})\}$. Moreover, we have already seen that $\{s: \tau(t_n+s) \le \gamma(t_{n^{*}})\} \subset \{s:\tau(s) \le C_4 (\gamma(t_{n^{*}}) +\tau(t_{n^{*}}))  \}$. Then dominated convergence shows that \eqref{eq:08111701}  holds. \\
For the proof of part (c) fix any $t \in \Gamma_d$. Using \eqref{eq:14061701} we see that
\begin{equation*}
\text{det} ( \phi(t-s)^{D-qB}-\phi(-s)^{D-qB})=\text{det}(\phi(\phi(-s)^{-E}(t-s))^{D-qB}-I_m)\cdot \text{det}(\phi(-s)^{D-qB}).
\end{equation*}
Since $\text{det}(\phi(-s)^{D-qB}) \ne 0$ for $s \ne 0$, \cref{20061701} (a) yields the assertion as long as
\begin{equation*}
0 \notin \sigma( \phi(t-s)^{D-qB}-\phi(-s)^{D-qB}) \quad \Leftrightarrow \quad 1 \notin \{\phi(\phi(-s)^{-E} (t-s))^{\lambda}: \lambda \in \sigma (D-qB)\}
\end{equation*} 
which is true for $s \notin \{0,t\}$ under the given assertion.
\end{proof}
\begin{example} \label{28061710}
For $\mu$ as before the assumptions of \cref{14061705} can always be fulfilled: Consider $E:=\text{diag}(a_1,...,a_d)$ with $a_1,...,a_d \ge 1$ and $D= c B$ for any $0<c < 1 / \Lambda_B$ such that $c \ne q=trace(E)$. Then we can use $\phi:\R^d \rightarrow [0,\infty)$, defined by $(x_1,...,x_d) \mapsto \sum_{j=1}^d |x_j|^{1/a_j}$. In particular we have $DB=BD$, where $\phi$ is $(1,E)$-admissible in the light of Corollary 2.12 in \cite{bms}. Note also that $d \ge 2$ is sufficient for $D-qB \in$ GL$(\R^m)$ in this example. 
\end{example}
\begin{remark} \label{27061703}
In view of Theorem 7 in \cite{HuMa82} (for $d=1$) the previous example, particularly the choice $D=cB$, should not appear unnatural. Additionally, we also get back the multivariate processes in \cite{MaMa92} with $d=1$ and $\phi(\cdot)=|\cdot|$, although their construction differs from ours. However, we should emphasize that the assumptions in \cref{14061705} coincide with those in \cite{LiXiao}, if $\mu$ is (rotationally symmetric) $\alpha$-stable. After all Proposition 6.1 in \cite{BiLa09} already suggests that every modification of $\mathbb{X}_{\phi,D}$ should have paths that are discontinuous a.s., at least for $d \ge 2$. But a rigorous proof of this conjecture would probably require an precise multivariate extension of chapter 10 in \cite{SaTaq94} for the operator-stable case.
\end{remark}
\section{Harmonizable representation} \label{chapter5}
In this section we want to generalize the random fields that are given by \eqref{eq:16051702} and its multivariate counterparts in \cite{LiXiao}. Therefore we have to replace $\mu$ by a full, strictly operator-stable distribution $\widetilde{\mu} \sim [\tilde{\gamma},\tilde{Q},\tilde{\nu}]$ on $\R^{2m}$ with log-characteristic function $\widetilde{\psi}$. Furthermore, we assume that an exponent of $\widetilde{\mu}$ is given by $B \oplus B$ for some $B \in \Li{m}$. Then the resulting ISRM $M$ (generated by $\widetilde{\mu}$ and the $d$-dimensional Lebesgue measure) can be identified with an $\C^m$-valued ISRM $\widetilde{M}$. Recall the notation in \cite{integral}, especially the mapping $\Xi(z)=(\text{Re }z,\text{Im }z ), z \in \C^m$ as well as Remark 3.8 there. Then we see that $\Xi(\widetilde{M})=M$, i.e. $M$ is the \textit{real associated} ISRM of $\widetilde{M}$. \\
Also note from section 5 in \cite{integral} that for functions $f:\R^d \rightarrow \text{L}(\C^m)$ integrals of the form $\text{Re } \int_{\R^d} f(s) \widetilde{M}(ds)$ can be defined in a reasonable way, even if $f \notin \mathcal{I}(\widetilde{M})$. That means that the integral $ \int_{\R^d} f(s) \widetilde{M}(ds)$ itself may not exist. Thus, there is a difference in general and we say that $f$ is partially integrable w.r.t $\widetilde{M}$, if $\text{Re } \int_{\R^d} f(s) \widetilde{M}(ds)$ exists at least.
\begin{theorem}[Harmonizable representation] \label{27061701}
Assume that $D \in Q(\R^m)$ and $E \in Q(\R^d)$ with $q:= \text{trace}(E)$ such that $\lambda_E> \Lambda_D$. Furthermore, let $\phi:\R^d \rightarrow [0,\infty)$ be a continuous and $E^{*}$-homogeneous function with $\phi(x)>0$ for $x \ne 0$. Then we have that the following stochastic integral exists for each $t \in \R^m$ in the partial sense:
\begin{equation} \label{eq:27061702}
\widetilde{X}_{\phi,D}:= \text{Re } \integral{\R^d}{}{\left( \text{e}^{\im \skp{t}{s}  }-1\right) \phi(s)^{-D} \phi(s)^{-qB}  }{\widetilde{M}(ds)}.
\end{equation}
We call the resulting random field $\widetilde{\mathbb{X}}_{\phi,D}:=\{\widetilde{X}_{\phi,D}(t):t \in \R^d\}$ a $(\phi,D)$-harmonizable representation (with respect to $\widetilde{M}$).
\end{theorem}
Differently from the moving-average representation, we observe that the condition $\lambda_E > \Lambda_D$ (independent of $B$) is identical to the corresponding one for the $S \alpha S$ case in \cite{LiXiao}. Moreover, in the $\alpha$-stable case we have $B=(1/ \alpha) I_m$ and hence \eqref{eq:27061702} is then equal the harmonizable representation in Theorem 2.6 in \cite{LiXiao}, since $D$ commutates with $I_m$. However, in general $\phi(s)^{-D} \phi(s)^{-qB} \ne \phi(s)^{-D-qB}$, but our method of proof requires that we use the form of the integrand as in \eqref{eq:27061702}. In fact, using $\phi(s)^{-D-qB}$ instead requires that $\Lambda_B-\lambda_B$ is very small which is a rather strong condition on $B$.  
\begin{proof} [Proof of \cref{27061701}]
Let $(\tau(\cdot),l(\cdot))$ denote the polar coordinates w.r.t. $E^{*}$. Observe that as in the proof of \cref{14061705} we have $m_{\phi}:=\min_{\theta \in S_{E^{*}}} \phi(\theta)>0$ as well as $M_{\phi}:=\max_{\theta \in S_{E^{*}}} \phi(\theta)>0$ and that $m_{\phi} \tau(s) \le \phi(s) \le M_{\phi \tau(s)}$ for every $s \in \R^d$. Fix any $t \ne 0$. In view of Proposition 5.10 in \cite{integral} it suffices to show that $f_t$, defined by 
\begin{equation*} 
f_t(s):= \begin{pmatrix}  (\cos \skp{t}{s}-1)  \phi(s)^{-D} \phi(s)^{-qB} & - \sin \skp{t}{s}  \phi(s)^{-D}\phi(s)^{-qB} \\ 0 & 0 \end{pmatrix}, \quad s \in \R^d,
\end{equation*}
belongs to $\mathcal{I}(M)$, i.e. is integrable w.r.t $M=\Xi(\widetilde{M})$. Using \cref{24051701} (iii) the following identity holds for $s \in \R^d$ and $u=(u_1,u_2) \in \R^{2m}$ arbitrary:
\allowdisplaybreaks
\begin{align}
\widetilde{\psi} (f_t(s)^{*}u)  &= \widetilde{\psi} \left(  \begin{pmatrix}  (\cos \skp{t}{s}-1)  \phi(s)^{-qB^{*}}  \phi(s)^{-D^{*} } u_1 \\ - \sin \skp{t}{s} \phi(s)^{-qB^{*}} \phi(s)^{-D^{*}} u_1 \end{pmatrix}   \right) \notag \\
&= \widetilde{\psi} \left(  \phi(s)^{-q\widetilde{B}^{*}} \begin{pmatrix}  (\cos \skp{t}{s}-1)   \phi(s)^{-D^{*} } u_1 \\ - \sin \skp{t}{s} \phi(s)^{-D^{*}} u_1 \end{pmatrix}   \right) \notag \\
&= \phi(s)^{-q} \, \widetilde{\psi} \left(   \begin{pmatrix}  (\cos \skp{t}{s}-1)   \phi(s)^{-D^{*} } u_1 \\ - \sin \skp{t}{s} \phi(s)^{-D^{*}} u_1 \end{pmatrix}   \right) \label{eq:04071711}.
\end{align}
Choose $0< \delta < 1/ \Lambda_B$ and $0< \eps < \min \{\lambda_D, (\lambda_E-\Lambda_D)/2 \}$ to define $\rho_1:=1/ \Lambda_B - \delta$ as well as $\rho_2:=1/ \lambda_B + \delta$ again. Writing $\phi(s)=\tau(s) \phi(l(s))$, the continuity of $\theta \mapsto \norm{\phi(\theta)^{-D}}$ together with \eqref{eq:17051701} yields constants $C_1,...C_4>0$ such that
\allowdisplaybreaks
\begin{align} 
\indikator{\{\tau(s) \ge 1\}}{s} \bignorm{\phi(s)^{-D}}^{\rho_j}  & \le C_{j} \tau(s)^{- (\lambda_{D}-\eps) \rho_j} , \notag \\
\indikator{\{\tau(s) < 1\}}{s} \bignorm{\phi(s)^{-D}}^{\rho_j}  & \le C_{j+2} \tau(s)^{-(\Lambda_{D}+\eps) \rho_j} \notag
\end{align}
for any $s \in \R^d$ and $j=1,2$. Then similar arguments and $\text{spec}(E^{*})=\text{spec}(E)$ provide a further constant $C_5>0$ such that
\begin{equation*} 
\forall \, 0<r \le 1 \, \, \forall \theta \in S_{E^{*}}: \quad \norm{r^{E^{*}}\theta}^{\rho_j} \le C_5 r^{\rho_j(\lambda_E-\eps)}, \quad j=1,2.
\end{equation*}
The Cauchy-Schwarz inequality and basic inequalities for $\sin(\cdot) / \cos(\cdot)$ ensure that
\begin{equation*} 
\forall \, x,y \in \R^d: \quad |\cos \skp{x}{y}-1|^{\rho_j} + |\sin \skp{x}{y}|^{\rho_j} \le C_6 \min \{1,\norm{x}^{\rho_j} \norm{y}^{\rho_j}\} , \quad j=1,2
\end{equation*}
with $C_6>0$ appropriate. Moreover, $\norm{x}^{\rho_j} \le C_7 (\norm{x_1}^{\rho_j}+\norm{x_2}^{\rho_j})$ holds for $j=1,2$ and any $x=(x_1,x_2) \in \R^{2m}$ with some $C_7>0$. Finally and in view of $\text{spec}(\widetilde{B})=\text{spec}(B)$ we can argue as before to see that there exists a constant $C_8>0$ such that
\begin{equation*} 
|\widetilde{\psi}(x)| \le C_8 (\norm{x}^{\rho_1} +\norm{x}^{\rho_2}) , \quad x \in \R^{2m}.
\end{equation*}
Putting things together and using Proposition 2.3 in \cite{bms} (for an corresponding finite measure $\sigma$ on $S_{E^{*}}$) we obtain the following estimates for any $u=(u_1,u_2) \in \R^{2m}$. Note that $trace(E^{*})=q$ and define $C_{9,j}(u)=C_{9,j}(u_1):= C_7 C_8 \norm{u_1}^{\rho_j} \max \{C_j,C_{j+2}\} m_{\phi}^{-1}$ for $j=1,2$. 
\allowdisplaybreaks
\begin{align}
& \integral{\R^d}{}{|\widetilde{\psi}(f_t(s)^{*}u)| }{ds} \notag \\
& \le C_8 \summe{j=1}{2}   \integral{\R^d}{}{\phi(s)^{-q}   \bignorm{\begin{pmatrix}  (\cos \skp{t}{s}-1)   \phi(s)^{-D^{*} } u_1 \\ - \sin \skp{t}{s}  \phi(s)^{-D^{*}} u_1 \end{pmatrix}}^{\rho_j}  }{ds}  \notag  \\
& \le C_7 C_8 \summe{j=1}{2}  \norm{u_1}^{\rho_j} \,   \int \limits_{\R^d}^{}  \phi(s)^{-q} \left( \, |\cos \skp{t}{s} -1|^{\rho_j}   +   |\sin \skp{t}{s}|^{\rho_j} \right)  \bignorm{\phi(s)^{-D}}^{\rho_j}  \, ds  \notag \\
& \le \summe{j=1}{2} C_{9,j} \integral{\{s: \tau(s) <1\} }{}{(|\cos \skp{t}{s} -1|^{\rho_j} +|\sin \skp{t}{s}|^{\rho_j} )  \tau(s)^{-q} \tau(s)^{-(\Lambda_D +\eps)\rho_j } }{ds}  \notag \\
& \quad + \summe{j=1}{2} C_{9,j} \integral{\{s: \tau(s) \ge 1\} }{}{(|\cos \skp{t}{s} -1|^{\rho_j} +|\sin \skp{t}{s}|^{\rho_j} )  \tau(s)^{-q} \tau(s)^{-(\lambda_D -\eps)\rho_j } }{ds}  \notag \\
& \le \summe{j=1}{2}  C_{9,j}C_6 \norm{t}^{\rho_j}  \int \limits_{0}^{1} r^{q-1} \int \limits_{S_{E^{*}}}^{} \norm{r^{E^{*}}\theta }^{\rho_j}  r^{-q-(\Lambda_D +\eps)\rho_j } \, \sigma(d \theta) \, dr  \notag \\
& \quad + \summe{j=1}{2} C_{9,j} C_6 \int \limits_{1}^{\infty} r^{q-1} \int \limits_{S_{E^{*}}}^{}  r^{-q-(\lambda_D -\eps)\rho_j } \, \sigma(d \theta) \, dr \notag \\
& \le \summe{j=1}{2} \left[ C_{9,j} C_6 C_5 \, \sigma(S_{E^{*}}) \norm{t}^{\rho_j} \int \limits_{0}^{1} r^{-1+ \rho_j(\lambda_E - \Lambda_D- 2 \eps)}\, dr + C_{9,j} C_6 \, \sigma(S_{E^{*}}) \int \limits_{1}^{\infty} r^{-1-(\lambda_D -\eps)\rho_j}  \, dr \right]\notag \\
&< \infty \notag
\end{align}
by the choice of $\eps$ and since $\rho_1,\rho_2>0$. Hence it remains to show the second condition of \cref{24051701} (iii). We fix $t \in \Gamma_d$ arbitrary ($t=0$ obvious) as well as any sequence $(u_n) \subset \R^{2m}$ with limit $u$ and show that
\begin{equation} \label{eq:28061705}
\integral{\R^d}{}{ \integral{\R^{2m}}{}{(\cos \skp{f_t(s)^{*}u_n}{x}-1) }{\widetilde{\nu}(dx)} }{ds} \rightarrow  \integral{\R^d}{}{ \integral{\R^{2m}}{}{(\cos \skp{f_t(s)^{*}u}{x}-1) }{\widetilde{\nu}(dx)} }{ds} 
\end{equation}
as $n \rightarrow \infty$. First of all, since $\norm{u_n} \le M$ for any $n \in \N$ and some $M>0$, we can check as performed in \eqref{eq:30051701} that
\begin{equation*}
 \integral{\R^{2m}}{}{(\cos \skp{f_t(s)^{*}u_n}{x}-1) }{\widetilde{\nu}(dx)} \rightarrow  \integral{\R^{2m}}{}{(\cos \skp{f_t(s)^{*}u}{x}-1) }{\widetilde{\nu}(dx)}. 
\end{equation*}
The first part of the proof revealed that $|\widetilde{\psi}(f_t(s)^{*}u)| \le g(t,u,s)$ holds for any $t \in \R^d, u \in \R^{2m}$ and $s \in \Gamma_d$, where
\begin{align*} 
g(t,u,s): & 	= C_{9,j}(u) C_6 \summe{j=1}{2}  \left[ \norm{t}^{\rho_j} \norm{s}^{\rho_j} \tau(s)^{-q-(\Lambda_D+\eps)\rho_j} \indikatorzwei{\{s: \tau(s)<1\}} + \tau(s)^{-q-(\lambda_D-\eps)\rho_j} \indikatorzwei{\{s: \tau(s)\ge 1\}} \right] 
\end{align*}
being integrable in $s$. Finally we use the idea of \eqref{eq:30051701} and observe that 
\begin{equation*}
\left| \, \, \integral{\R^{2m}}{}{ (\cos \skp{f_t(s)^{*}u_n}{x}-1) }{\widetilde{\nu}(dx)} \right|  \le \left| \widetilde{\psi} (f_t(s)^{*}u_n) \right| \le g(t,u_n,s) \le g(t,\widetilde{u},s)
\end{equation*}
holds for any $n \in \N$ and $s \in \Gamma_d$ with $\widetilde{u}:=(M,0,...,,0) \in \R^{2m}$ for example (see the definition of $C_{9,j}$). Then dominated convergence yields the assertion.
\end{proof}
As in \cref{27061705} above for the moving-average representation we now collect some basic properties of the harmonizable representation obtained in \cref{27061701} above. Observe that there are some differences between both representations and that 
\begin{equation}  \label{eq:29061710}
R(z):= \frac{1}{|z|} \begin{pmatrix} (\text{Re }z) \, I_m & (\text{Im }z) \,   I_m \\ - (\text{Im }z)  \,  I_m & (\text{Re }z)  \,  I_m \end{pmatrix} \in \text{GL}(\R^{2m})
\end{equation}
is a rotation matrix for any $z \in \C \setminus \{0\}$ and $m \in \N$, i.e. $\text{det}(R(z))=1$ and $R(z)^{-1}=R(z)^{*}$. In particular the elements of the set
\begin{equation*}
\mathcal{T}(2m):= \left \{ \begin{pmatrix} (\cos \beta) I_m & (\sin \beta) I_m \\ - (\sin \beta) I_m & (\cos \beta) I_m \end{pmatrix} : \beta \in [0,2 \pi)  \right \}, \quad m \in \N
\end{equation*}
are rotation matrices. Then $\mathcal{T}(2m)$ is an Abelian group that equals the rotation group SO$(2)$ for $m=1$ and is isomorphic to SO$(2)$ for $m \ge 2$. According to Definition 2.6.2 in \cite{SaTaq94} (for $m=1$) we say that a distribution $\mu$ on $\R^m$ is \textit{isotropic}, if $(A \mu)=\mu$ for every rotation matrix $A \in \text{SO}(m)$. Observe that for the harmonizable representation in the $\alpha$-stable case considered in \cite{BiLa09}, \cite{bms} and \cite{LiXiao} one always uses a complex-valued isotropic $\alpha$-stable random measure. Now a similar, but generally weaker condition is needed in our situation in order to get stationary increments, see part (b) of the following proposition.  
\begin{prop} 
Assume that the assumptions of \cref{27061701} are fulfilled. 
\begin{itemize}
\item[(a)] If $DB=BD$, then $\widetilde{\mathbb{X}}_{\phi,D}$ is $(E,D)$-OSS and strictly Operator-stable with exponent $B$.
\item[(b)] If $(A \widetilde{\mu})=\widetilde{\mu}$ for any $A \in \mathcal{T}(2m)$, then $\widetilde{\mathbb{X}}_{\phi,D}$  has stationary increments.
\item[(c)] $\widetilde{\mathbb{X}}_{\phi,D}$ is full and stochastically continuous.
\end{itemize}
\end{prop}
\begin{proof}
If $DB=BD$, the characteristic function of $(X_{\phi,D}(t_1),...,X_{\phi,D}(t_k))$ is given by  
\begin{equation} \label{eq:07071701}
\R^{k \cdot m} \ni u=(u_1,...,u_k) \mapsto \exp \left( \, \integral{\R^d}{}{  \widetilde{\psi}\left( \begin{pmatrix} \summezwei{j=1}{k} (\cos \skp{t_j}{s} -1) \phi(s)^{-D^{*}-qB^{*}}u_j \\ - \summezwei{j=1}{k}  \sin \skp{t_j}{s} \phi(s)^{-D^{*}-qB^{*}}u_j \end{pmatrix}  \right)}{ds} \right)
\end{equation}
for any $k \in \N$ and $t_1,...,t_k \in \R^d$ (see Example 3.7 and Remark 5.12 in \cite{integral}). Since $\widetilde{B}^{*}=B^{*} \oplus B^{*}$, we can use \cref{29061701} and linearity (see Proposition 5.3 in \cite{integral}) to argue as in the corresponding proof for the operator-stability of \cref{27061705}. The claimed operator-self-similarity of $\widetilde{\mathbb{X}}_{\phi,D}$ can also be observed similar as before, hence by a change of variables again. Merely note that $\phi$ is $E^{*}$-homogeneous such that $\text{tr}(E^{*})=q$ yields
\allowdisplaybreaks
\begin{align*}
& \integral{\R^d}{}{  \widetilde{\psi}\left( \begin{pmatrix} \summezwei{j=1}{k} (\cos \skp{ t_j}{c^{E^{*}} s} -1) \phi(s)^{-D^{*}-qB^{*}}u_j \\ - \summezwei{j=1}{k}  \sin \skp{ t_j}{c^{E^{*}} s} \phi(s)^{-D^{*}-qB^{*}}u_j \end{pmatrix}  \right)}{ds} \\
&\qquad =  \integral{\R^d}{}{ c^{-q} \cdot \widetilde{\psi}\left( \begin{pmatrix} \summezwei{j=1}{k} (\cos \skp{ t_j}{s} -1)c^{D^{*}+q B^{*}}  \phi(s)^{-D^{*}-qB^{*}}u_j \\ - \summezwei{j=1}{k}  \sin \skp{ t_j}{s} c^{D^{*}+q B^{*}}  \phi(s)^{-D^{*}-qB^{*}}u_j \end{pmatrix}  \right)}{ds}
\end{align*}
for any $c>0$. For the proof of (b) observe that for any $t,h \in \R^d$ we have 
\begin{equation*}
\widetilde{X}_{\phi,D}(t+h) - \widetilde{X}_{\phi,D}(h)=\text{Re } \int_{\R^d}^{}\text{e}^{\im \skp{h}{s}} \left (\text{e}^{\im \skp{t}{s}}-1 \right) \phi(s)^{-D} \phi(s)^{-qB} \, \widetilde{M}(ds) \quad \text{a.s.} 
\end{equation*}
A simple calculation together with Remarkt 5.12 in \cite{integral} shows that the characteristic function of $\widetilde{X}_{\phi,D}(t+h) - \widetilde{X}_{\phi,D}(h)$ is given by $\exp(\int_{\R^d} \widetilde{\psi} (\zeta(h,t,s,u)) \, ds)$ for any $u \in \R^m$, where
\allowdisplaybreaks
\begin{align*}
\zeta(h,t,s,u):&= \begin{pmatrix}\left( \cos \skp{h}{s} (\cos \skp{t}{s}-1) -\sin \skp{h}{s} \sin \skp{t}{s} \right) \phi(s)^{-qB^{*}} \phi(s)^{-D^{*}} u \\ - \left(\cos \skp{h}{s} \sin \skp{t}{s}+ (\cos \skp{t}{s}-1) \sin \skp{h}{s} \right)  \phi(s)^{-qB^{*}} \phi(s)^{-D^{*}} u  \end{pmatrix} \\
&= R(\text{e}^{\im \skp{h}{s}})  \begin{pmatrix} (\cos \skp{t}{s}-1)I_m \\ - \sin \skp{t}{s}I_m \end{pmatrix} \phi(s)^{-qB^{*}} \phi(s)^{-D^{*}}u \\
&= R(\text{e}^{\im \skp{-h}{s}})^{*}  \begin{pmatrix} (\cos \skp{t}{s} -1) \phi(s)^{-qB^{*}} \phi(s)^{-D^{*}} u \\ -   \sin \skp{t}{s} \phi(s)^{-qB^{*}} \phi(s)^{-D^{*}} u\end{pmatrix} ,
\end{align*}
using \eqref{eq:29061710}. Observe that the assumption in (b) is equivalent to $\widetilde{\psi}(A^{*}x) =\widetilde{\psi}(x)$ for any $A \in \mathcal{T}(2m)$ and $x \in \R^{2m}$, see Lemma 3.1.10 in \cite{thebook} for example. Thus we obtain (b) similar to the corresponding proof for the moving-average representation. The details are left to the reader. We now proof part (c). Concerning the stochastic continuity of $\widetilde{\mathbb{X}}_{\phi,D}$ we fix $t_0 \in \R^d$ and $u \in \R^m$ as well as some null sequence $(t_n) \subset \R^d$ (particularly we have $\norm{t_n} \le M$ for some $M>0$). In view of Remark 5.12 in \cite{integral} and by what we have shown before, it suffices to show that $\int_{\R^d} \widetilde{\psi} (\zeta(t_0,t_n,s,u)) \, ds \rightarrow 0$ as $n \rightarrow \infty$. Obviously the integrand converges pointwise to $\widetilde{\psi}(0)=0$. Then we can write 
\begin{equation*}
 \widetilde{\psi} (\zeta(t_0,t_n,s,u)) = \phi(s)^{-q} \cdot \widetilde{\psi} \left (R(\text{e}^{\im \skp{t_0}{s}}) \begin{pmatrix} (\cos \skp{t_n}{s} -1)  \phi(s)^{-D^{*}} u \\ -   \sin \skp{t_n}{s}  \phi(s)^{-D^{*}} u\end{pmatrix} \right)
 \end{equation*}
as above. Obviously $\norm{R(\text{e}^{\im \skp{t_0}{s}})x}=\norm{x}$ for any $s \in \R^d$ and $x \in \R^{2m}$. Now argue as in the proof of \cref{27061701} to see that $| \widetilde{\psi} (\zeta(t_0,t_n,s,u))| \le g (\widetilde{t},\widetilde{u},s)$ with $\widetilde{t}:=(M,0,...,0) \in \R^d$ and $\widetilde{u}:=(u,0) \in \R^{2m}$ due to the definition of $g$. Then dominated convergence gives the assertion. Finally we have to show the fullness of $\widetilde{X}_{\phi,D}(t)$ for $t \in \Gamma_d$ arbitrary and immediately check that at least the real- or the imaginary part of the integrand in \eqref{eq:27061702} is invertible as long as $s \notin K:=\{0\} \cup \{s: \text{e}^{\im \skp{t}{s}}-1=0\}$. Obviously $K^c$ has positive Lebesgue measure and the fullness follows from \cref{29061701}.
\end{proof}
It seems to be open if the assumption in part (b) is also necessary for $\widetilde{\mathbb{X}}_{\phi,D}$ to have stationary increments. However, we get a lot of non-trivial examples again.
\begin{example}
For any full and strictly operator-stable generator $\widetilde{\mu}$ with exponent $B \oplus B$ the requirements of \cref{27061701} can always be fulfilled such that $DB=BD$. More precisely, we consider $E=\text{diag}(a_1,...,a_d)$ with $a_j>0$ and $D:=\gamma B$ for some $0< \gamma < \min \{a_1,...,a_d\}$. Then $\phi$ can be chosen exactly as in \cref{28061710}.
\end{example}
Finally we want to consider a special case which is still purely operator-stable and which is inspired by the $\alpha$-stable case, considered in \cite{BiLa15} and \cite{Soen16}. Hence for the rest of this section we assume that the assumptions of \cref{27061701} are fulfilled. Additionally, we assume that $DB=BD$. Furthermore let $\widetilde{\mu}$ be isotropic and suppose that 
\begin{equation*}
B=B^{*}=\text{diag}(\alpha_1^{-1},...,\alpha_m^{-1})\quad \text{with $\alpha_1,...,\alpha_m \in (0,2)$.}
\end{equation*}
For convenience we denote the corresponding harmonizable representation from \cref{27061701} by $\mathbb{Y}=\{Y(t):t \in \R^d\}$ with $Y(t)=(Y_1(t),...,Y_m(t))$ and let $q=trace(E)$ again. The first step is to understand the resulting component fields $\mathbb{Y}_j:=\{Y_j(t): t \in \R^d\}$. In the following let $f_j(s):= |\widetilde{\psi}((\phi(s)^{-D^{*}-q B^{*}}e_j,0)) |^{1 / \alpha_j}$ for $s \in \R^d$ and $j=1,...,m$.
\begin{prop} \label{05071701}
Let $n \in \N$ be arbitrary. Then the following relation holds for every $1 \le j \le m, u=(u_1,...,u_n) \in \R^n$ and $t_1,...,t_n \in \R^d$:
\begin{equation} \label{eq:30061701}
\bigerwartung{\text{e}^{\im \sum_{k=1}^{n} Y_j(t_k)u_k}} = \exp \left( - \integral{\R^d}{}{|\sum_{k=1}^{n} (\text{e}^{\im \skp{t_k}{s} } -1)u_k |^{\alpha_j} |f_j(s)|^{\alpha_j} }{ds} \right).
\end{equation}
\end{prop}
\begin{proof}
Since $\widetilde{\mu}$ is isotropic, it is especially symmetric and hence $\widetilde{\psi}(\cdot) \in (-\infty,0]$. Let $z:=z(s):= \summezwei{k=1}{n} (\text{e}^{\im \skp{t_k}{s}}-1) u_k$ and $A:=A(s)=\phi(s)^{-D^{*}-qB^{*}}$ for $j,u$ and $t_k$ as above. Then the following calculation holds for $z(s) \ne 0$, since $\widetilde{\mu}$ is isotropic and since $DB=BD$.
\allowdisplaybreaks
\begin{align}
&\widetilde{\psi} \left(   \begin{pmatrix} (\summezwei{k=1}{n} (\cos \skp{t_k}{s} -1)u_k) \,  \phi(s)^{-D^{*}-qB^{*}} e_j \\ - (\summezwei{k=1}{n}  \sin \skp{t_k}{s} u_k) \,  \phi(s)^{-D^{*}-qB^{*}}e_j \end{pmatrix} \right)  \notag \\
& \qquad = \widetilde{\psi} \left( \begin{pmatrix} (\text{Re }z) A & (\text{Im }z) A \\ -(\text{Im }z) A & (\text{Re }z) A  \end{pmatrix}  \begin{pmatrix} e_j \\0 \end{pmatrix} \right)  \notag \\
& \qquad = \widetilde{\psi} \left( \frac{1}{|z|} \begin{pmatrix} (\text{Re }z) I_m & (\text{Im }z) I_m \\ -(\text{Im }z) I_m & (\text{Re }z) I_m  \end{pmatrix} \begin{pmatrix} A & 0 \\ 0 & A \end{pmatrix}  \begin{pmatrix} |z| e_j \\ 0 \end{pmatrix} \right)  \notag \\
& \qquad = \widetilde{\psi} \left(  \begin{pmatrix} A & 0 \\ 0 & A \end{pmatrix}  \begin{pmatrix} |z| e_j \\  0 \end{pmatrix} \right)  \notag \\
& \qquad = \widetilde{\psi} \left(  \begin{pmatrix} A & 0 \\ 0 & A \end{pmatrix}  (|z|^{\alpha_j})^{\widetilde{B}^{*}} \begin{pmatrix}  e_j \\  0 \end{pmatrix} \right)  \notag \\
& \qquad = \widetilde{\psi} \left( (|z|^{\alpha_j})^{\widetilde{B}^{*}}   \begin{pmatrix} A & 0 \\ 0 & A \end{pmatrix}  \begin{pmatrix}  e_j \\  0 \end{pmatrix} \right) \notag \\
& \qquad = |z|^{\alpha_j} \cdot \widetilde{\psi} \left(   \begin{pmatrix} \phi(s)^{-D^{*}-qB^{*} } e_j \\ 0 \end{pmatrix} \right). \notag 
\end{align}
Here we used that for diagonal matrices $C=\text{diag}(c_1,...,c_m)$ we have $r^C e_j=r^{c_j}e_j$. Furthermore, the previous identity is trivially true for $z(s)=0$, since $\widetilde{\psi}(0)=0$. Then the assertion follows from $Y_j(t_k) u_k = \skp{Y(t_k)}{u_k e_j}$ and \eqref{eq:07071701}.
\end{proof}
Subsequently the separation of the integrand in \eqref{eq:30061701} will turn out to be crucial in order to use the results in \cite{BiLa15}. In preparation we need sharp bounds on $f_j$. Recall (see Theorem 2.1.16 in \cite{thebook} for example) that an operator $J \in \Li{m}$ has a \textit{(real) Jordan form}, if $J=\text{diag } (J_1,...,J_p)$ for some $1 \le p \le m$, where $J_l$ can be represented as
\begin{equation} \label{eq:04071701}
J_l= \begin{pmatrix}  \lambda_l & 1 & & & \\ & \lambda_l & 1 & & \\ & & \ddots & \ddots &  \\ & & & \ddots & 1 \\ & & & & \lambda_l \end{pmatrix}
\end{equation}
for any $1 \le l \le p$ and some $\lambda_l \in \R$ (\textit{real Jordan block}) or as
\begin{equation} \label{eq:04071702}
J_l=\begin{pmatrix}  \Lambda_l & I_2 & & & \\ & \Lambda_l & I_2 & & \\ & & \ddots & \ddots &  \\ & & & \ddots & I_2 \\ & & & & \Lambda_l \end{pmatrix} \quad \text{with} \quad \Lambda_l=\begin{pmatrix} a_l & -b_l \\ b_l & a_l\end{pmatrix}
\end{equation}
for some $\lambda_l:=a_l+ \im \, b_l \in \C$ (\textit{complex Jordan block}), respectively. Hence $\lambda_1,...,\lambda_p$ are the (not necessarily different) eigenvalues of $J$ and we denote by $n(l)$ the size of the $l$-th block. Particularly, we know that for any $A \in \Li{m}$ there exists a $P \in$ GL$(\R^m)$  such that $\widetilde{A}:=P^{-1}A P$ has a (real) Jordan form. Finally we have $\lambda_l \in \text{spec}(A)$ for every block as in \eqref{eq:04071701} or \eqref{eq:04071702}. Although the order of blocks can be arbitrary (by permuting the columns of $P$), $AB=BA$ does not imply $\widetilde{A}B=B \widetilde{A}$ in general.

\begin{lemma} \label{04071703}
Fix $L>0$ and $1 \le j \le m$. Then we have:
\begin{itemize}
\item[(i)] For any $0< \eps < \lambda_D$ there exists a constant $K_1>0$ such that
\begin{equation*}
f_j(s) \le K_1 \tau_{E^{*}}(s)^{-(\lambda_D-\eps)-1/Ê\alpha_j}, \quad \norm{s} \ge L.
\end{equation*}
\item[(ii)] Assume that $D=\text{diag}(D_1,...,D_p)$ for some $1 \le p \le m$ with $D_l$ as in \eqref{eq:04071701} or \eqref{eq:04071702} for $1 \le l \le p$ and define $j_{+} \in \{1,...,p\}$ in virtue of $\sum_{l=1}^{j_{+}-1} n(l) <j \le \sum_{l=1}^{j_{+}} n(l)$. Then, for any $0< \eps < \text{Re } \lambda_{j_{+}}$, there exists a constant $K_2>0$  such that
\begin{equation*}
f_j(s) \le K_2 \tau_{E^{*}}(s)^{-(\text{Re } \lambda_{j_{+}}-\eps)-1/Ê\alpha_j}, \quad \norm{s} \ge L.
\end{equation*}
\end{itemize}
\end{lemma}
\begin{proof}
As in \eqref{eq:04071711} we observe via $DB=BD$ that $f_j(s)=\phi(s)^{-q/ \alpha_j} |\widetilde{\psi} ((\phi(s)^{-D^{*}}e_j,0)) |^{1/ \alpha_j}$ for any $s \in \Gamma_d$. At the same time we can use the polar coordinates w.r.t. $\widetilde{B}=\widetilde{B}^{*}$ and verify the existence of a constant $C_1>0$ with $|\widetilde{\psi}(x)| \le C_1 \tau_{\widetilde{B}}(x)$ for any $x \in \R^{2m}$. Thus we have (see the notation in the proof of \cref{27061701}):
\allowdisplaybreaks
\begin{align} 
f_j(s) &\le C_1^{1/ \alpha_j} \phi(s)^{-q/ \alpha_j} \, \tau_{\widetilde{B}} \left(  \begin{pmatrix} \phi(s)^{-D^{*}} e_j \\ 0 \end{pmatrix}  \right)^{1/ \alpha_j} \notag \\
& \le C_1^{1/ \alpha_j} m_{\phi}^{-q / \alpha_j} \, \tau_{E^{*}}(s)^{-q / \alpha_j} \, \tau_{B} \left(   \phi(s)^{-D^{*}} e_j   \right)^{1/ \alpha_j} , \quad s \in \Gamma_d ,\notag
\end{align}
since $\tau_B(y)=\tau_{\widetilde{B}}((y,0))$ for any $y \in \R^m$ which can be checked easily with the remarks at the beginning of chapter 2. Now fix $\eps>0$ as in (i) or (ii) and notice that it remains to show that there exist constants $C_3,C_4>0$ with
\begin{equation*}
\text{(i)}\, \tau_B(\phi(s)^{-D^{*}} e_j) \le C_3 \tau_{E^{*}}(s)^{-\alpha_j(\lambda_D-\eps)}, \qquad \text{(ii)} \, \tau_B(\phi(s)^{-D^{*}} e_j) \le C_4 \tau_{E^{*}}(s)^{-\alpha_j(\text{Re } \lambda_{j_+}-\eps)}
\end{equation*}
for any $\norm{s} \ge L$, respectively. Concerning (i) recall that $D^{*}B=B D^{*}$ and that $\tau_B$ is a $B$-homogeneous function. Hence we obtain thanks to the form of $B$ again that
\allowdisplaybreaks
\begin{align*}
\tau_B(\phi(s)^{-D^{*}} e_j) &= \phi(s)^{-\alpha_j(\lambda_D - \eps)} \tau_B\left( \phi(s)^{\alpha_j(\lambda_D- \eps) B } \phi(s)^{-D^{*}} e_j \right) \\
&= \phi(s)^{-\alpha_j(\lambda_D - \eps)} \tau_B\left(  \phi(s)^{-D^{*}}  \phi(s)^{\alpha_j(\lambda_D- \eps) B } e_j \right) \\
&= \phi(s)^{-\alpha_j(\lambda_D - \eps)} \tau_B\left(  \phi(s)^{-D^{*}}  \phi(s)^{\lambda_D- \eps  } e_j \right).
\end{align*}
Since $\text{spec}(D)=\text{spec}(D^{*})$ we have
\begin{equation*}
\bignorm{ \phi(s)^{-D^{*}}  \phi(s)^{\lambda_D- \eps  } e_j}  \le \phi(s)^{\lambda_D-\eps} \, \norm{\phi(s)^{-D^{*}}} \le C_0 \, \phi(s)^{\lambda_D-\eps} \phi(s)^{- (\lambda_D-\eps)}=  C_0
\end{equation*}
for any $\norm{s} \ge L$ and some constant $C_0$ due to \eqref{eq:17051701}. Since $\tau_B$ is continuous, it follows that
\begin{equation*}
C_3:= m_{\phi}^{- \alpha_j (\lambda_D - \eps)} \, \sup \{\tau_B (\phi(s)^{(\lambda_D- \eps) I_m -D^{*}} e_j): \norm{s} \ge L \}<\infty. 
\end{equation*}
Similarly in the situation of (ii) we merely have to verify that $\phi(s)^{\text{Re } \lambda_{j_{+}}-\eps} \, \norm{\phi(s)^{-D^{*}} e_j} $ is bounded for $\norm{s} \ge L$. For this purpose the definition of $j_{+}$ and  $D^{*}=\text{diag}(D_1^{*},...,D_p^{*})$ obviously permit a slight extension of Corollary 2.2 (b) in \cite{Soen16}, leading to 
\begin{equation*}
\norm{\phi(s)^{-D^{*}} e_j} \le C_0'  \,\phi(s)^{- \text{Re } \lambda_{j_{+}} + \eps}, \quad \norm{s} \ge L
\end{equation*}
for a corresponding constant $C_0'$.
\end{proof}
The following result generalizes Proposition 4.1 in \cite{Soen16}.
\begin{theorem} \label{03071701}
Assume that the previous assumptions are fulfilled and let $\beta_j:= \text{Re } \lambda_{j_+}$ for $j=1,...,m$, if $D$ has a Jordan form as in \cref{04071703} (ii) and $\beta_j:= \lambda_D$ else. Then we have: For any non-empty, compact subset $K \subset \R^d$ as well as for any $\delta>0$ and $0<\eps< \min \beta_j$ there exists a modification $\mathbb{Y}^{*}:=\{Y^{*}(t)=(Y_1^{*}(t),...,Y^{*}_m(t)): t \in K \}$ of $\mathbb{Y}$ (on $K$), i.e. $Y(t)=Y^{*}(t)$ a.s. for every $t \in K$, such that 
\begin{equation} \label{eq:03071702}
\underset{\substack{s,t \in K \\ s \ne t} }{\sup} \,  \frac{|Y_j^{*}(s) - Y_j^{*}(t)| }{\tau_E(s-t)^{\beta_j- \eps} [\log (1+\tau_E(s-t)^{-1})]^{\delta+\frac{1}{2}+\frac{1}{\alpha_j}  } } < \infty \quad \text{a.s.}
\end{equation}
for $j=1,...,m$. Particularly, for any $0<\eps< \min \beta_j$ there exists a modification $\mathbb{Y}^{*}$ of $\mathbb{Y}$ on $K$ such that we have:
\begin{equation*}
\forall s,t \in K : \qquad |Y_j^{*}(s) - Y_j^{*}(t)|  \le A_j \, \tau_E(s-t)^{\beta_j - \eps}  \quad \text{a.s.}
\end{equation*}
 for $j=1,...,m$ and appropriate $[0,\infty)$-valued random variables $A_j$.
\end{theorem}
\begin{proof}
For $j=1,...,m$ let $M_{\alpha_j}$ be the $\C$-valued ISRM that we obtain by identification with the $\R^{2}$-valued one which is generated by the $d$-dimensional Lebesgue measure and the distribution $\mu_j$ with Fourier transform $\R^2 \ni x \mapsto \exp(-\norm{x}^{\alpha_j})$. That is $M_{\alpha_j}$ is a complex isotropic $\alpha_j$-stable random measure in the sense of \cite{SaTaq94}. Similar as in the proof of \cref{05071701} we then see that
\begin{equation*}
\left|  (\text{e}^{\im \skp{t}{s}}-1) f_j(s) \right|^{\alpha_j}  = \left| \widetilde{\psi} \left(   \begin{pmatrix} (\cos \skp{t}{s} -1) \phi(s)^{-qB^{*}} \phi(s)^{-D^{*}}  e_j \\ -  \sin \skp{t}{s} \phi(s)^{-qB^{*}} \phi(s)^{-D^{*}} e_j \end{pmatrix} \right) \right|.
\end{equation*}
Furthermore, the proof of \cref{27061701} revealed that $\left|  (\text{e}^{\im \skp{t}{s}}-1) f_j(s) \right|^{\alpha_j}$ is integrable (in $s$) for every $t \in \R^d$. By definition of the Frobenius norm $\norm{\cdot}_F$ we also note that 
\begin{equation*}
\bignorm{\begin{pmatrix} (\cos \skp{t}{s}-1)f_j(s) & - \sin \skp{t}{s} f_j(s) \\  \sin \skp{t}{s} f_j(s) & (\cos \skp{t}{s}-1)f_j(s)    \end{pmatrix}}^{\alpha_j}_F = 2^{\alpha_j/2} \,  \left|  (\text{e}^{\im \skp{t}{s}}-1) f_j(s) \right|^{\alpha_j}.
\end{equation*}
Overall, in view of \cref{30051702} and the results in \cite{inteâgral}, this proves that the $\C$-valued random fields $\mathbb{Z}_j'=\{Z_j'(t): t \in \R^d\}$, defined via
\begin{equation*}
Z_j'(t):= \integral{\R^d}{}{ (\text{e}^{\im \skp{t}{s}}-1) f_j(s)}{M_{\alpha_j}(ds)}, \quad t \in \R^d,
\end{equation*}
exist for $j=1,...,m$ (on appropriate probability spaces). And except for a negligible, multiplicative constant the fields $\mathbb{Z}_j'$ have a representation as in (17) in \cite{BiLa15} (see also (6.3.1) in \cite{SaTaq94}). Hence Proposition 5.1 in \cite{BiLa15} can be applied (together with \cref{04071703}) and yields: For $K, \delta$ and $\eps$ as mentioned above (i.e. $0<\beta_j-\eps<\lambda_E$) there exists a modification $\mathbb{Z}_j=\{Z_j(t): t \in K \}$ of $\mathbb{Z}_j'$ (for $j=1,...,m$) such that
\begin{equation}  \label{eq:05071705}
\underset{\substack{s,t \in K \\ s \ne t} }{\sup} \,  \frac{|Z_j(s) - Z_j(t)| }{\tau_E(s-t)^{\beta_j- \eps} [\log (1+\tau_E(s-t)^{-1})]^{\delta+\frac{1}{2}+\frac{1}{\alpha_j}  } } < \infty \quad \text{a.s.}
\end{equation}
Then \eqref{eq:05071705} also holds with $\text{Re } Z_j$ in the numerator and $\text{Re } \mathbb{Z}_j$ is a modification of $\text{Re } \mathbb{Z}'_j$ as well. This allows to finish the proof quite similar to Proposition 4.1 in \cite{Soen16} as well as to Proposition 5.1 in \cite{BiLa15}, because $\mathbb{Y}_j $ and $\text{Re } \mathbb{Z}_j'$ have the same finite-dimensional distributions. To verify this we fix $1 \le j \le m, t_1,...,t_n \in \R^d$ and $u_1,...,u_n \in \R$ for some $n \in \N$. Then we obtain with \cref{05071701} and Remark 5.12 in \cite{integral} that
\allowdisplaybreaks
\begin{align*}
& \bigerwartung{\text{e}^{\im \summezwei{k=1}{n} (\text{Re } Z'_j (t_k)) u_k }} \\
& = \exp \left(- \integral{\R^d}{}{\bignorm{\begin{pmatrix} \summezwei{k=1}{n} (\cos\skp{t_k}{s}-1) f_j(s) u_k \\  - \summezwei{k=1}{n} \sin \skp{t_k}{s} f_j(s) u_k \end{pmatrix} }^{\alpha_j} }{ds}  \right) \\
&  = \exp \left(- \integral{\R^d}{}{ \left(  \left(\summezwei{k=1}{n} (\cos\skp{t_k}{s}-1)u_k \right)^2 f_j(s)^2 + \left (\summezwei{k=1}{n} \sin \skp{t_k}{s} u_k \right)^2 f_j(s)^2  \right)^{\alpha_j /2 } }{ds}  \right) \\
&  = \exp \left(- \integral{\R^d}{}{ | \summezwei{k=1}{n} \left(\text{e}^{\im \skp{t_k}{s} } -1 \right)u_k |^{\alpha_j} |f_j(s)|^{\alpha_j} }{ds}  \right)  \\
& = \bigerwartung{\text{e}^{\im \summezwei{k=1}{n} Y_j(t_k) u_k }}.
\end{align*}
For the additional statement of this theorem merely use \eqref{eq:03071702} and write $\tau_E(s-t)^{\beta_j- \frac{\eps}{2}}= \tau_E(s-t)^{\beta_j- \eps} \tau_E(s-t)^{ \frac{\eps}{2}}$ since the growth of the logarithm is more slowly than that of any polynomial with positive degree.
\end{proof}
\begin{remark}
\begin{itemize}
\item[(a)] By a slight refinement of the previous proof (and the use of Proposition 5.1 in \cite{BiLa15}) it is possible to define a modification that lives on whole $\R^d$, fulfilling \eqref{eq:03071702} accordingly for any $K,\delta$ and $\eps$ as mentioned above.
\item[(b)] A combination of \cref{03071701} and the results in \cite{Soen17} allows to calculate the Hausdorff dimension of the image and the graph of the harmonizable representation, at least for the present case. Then we should get back the results of Theorem 5.1 in \cite{Soen16} which will be independent of $B$ again, i.e. they only depend on $\text{spec}(E)$ and $\text{spec}(D)$.
\end{itemize}
\end{remark}
\section*{Acknowledgement}
The results of this paper are part of the first author's PhD-thesis (see \cite{Diss}), written under the supervision of the second named author. Morevover, the authors would like to thank Ercan S\"onmez for fruitful discussions leading to \cref{03071701}.


\end{document}